\documentclass[12pt]{article}
\usepackage{amssymb,bm}
\usepackage{amsmath}
\usepackage{amsthm}
\usepackage{graphicx}
\usepackage[active]{srcltx} 
\usepackage{hyperref} 
\hypersetup{pdfborder=0 0 0}
\usepackage[top=1.05in,left=1.05in,right=1.05in,bottom=1.05in]{geometry}

\newtheorem{theorem}{{\sc Theorem}}[section]

\newtheorem{lemma}[theorem]{{\sc Lemma}}

\newcommand{\bb}[1]{\mathbb{ #1}}

\bmdefine\Bone{1}

\newcommand{\defeq}{{\buildrel\rm def\over=}}
\newcommand{\wk}[1]{{\buildrel #1\over\rightharpoonup}\:}
\newcommand{\weak}{\rightharpoonup\:}

\newcommand{\eqv}{\Longleftrightarrow}
\newcommand{\bra}[1]{\overline{#1}}

\newcommand{\hf}{\displaystyle\frac{1}{2}}
\newcommand{\nth}[1]{\displaystyle\frac{1}{#1}}

\newcommand{\Md}{\partial}
\renewcommand{\Hat}[1]{\widehat{#1}}
\newcommand{\Tld}[1]{\widetilde{#1}}

\def\XXint#1#2#3{{\setbox0=\hbox{$#1{#2#3}{\int}$ }
\vcenter{\hbox{$#2#3$ }}\kern-.6\wd0}}

\newcommand\myatop[2]{\genfrac{}{}{0pt}{}{#1}{#2}}

\newcommand{\re}{\Re\mathfrak{e}}

\newcommand{\lims}{\mathop{\overline\lim}}
\newcommand{\limi}{\mathop{\underline\lim}}

\newcommand{\rhs}{right-hand side}
\newcommand{\lhs}{left-hand side}

\newcommand{\WLOG}{without loss of generality}

\newcommand{\IFF}{if and only if }

\newcommand{\Ga}{\alpha}
\newcommand{\Gb}{\beta}
\newcommand{\Gd}{\delta}
\newcommand{\Ge}{\epsilon}
\newcommand{\eps}{\varepsilon}
\newcommand{\Gve}{\varepsilon}

\newcommand{\Gg}{\gamma}

\newcommand{\Gl}{\lambda}

\newcommand{\Gth}{\theta}

\newcommand{\Gs}{\sigma}

\newcommand{\Gz}{\zeta}
\newcommand{\GD}{\Delta}

\newcommand{\GG}{\Gamma}

\newcommand{\GL}{\Lambda}

\newcommand{\GO}{\Omega}

\bmdefine\BGa{\alpha}
\bmdefine\BGb{\beta}
\bmdefine\BGd{\delta}
\bmdefine\BGe{\epsilon}
\bmdefine\BGve{\varepsilon}
\bmdefine\BGf{\phi}
\bmdefine\BGvf{\varphi}
\bmdefine\BGg{\gamma}
\bmdefine\BGc{\chi}
\bmdefine\BGi{\iota}
\bmdefine\BGk{\kappa}
\bmdefine\BGl{\lambda}
\bmdefine\BGn{\eta}
\bmdefine\BGm{\mu}
\bmdefine\BGv{\nu}
\bmdefine\BGp{\pi}
\bmdefine\BGth{\theta}
\bmdefine\BGvth{\vartheta}
\bmdefine\BGr{\rho}
\bmdefine\BGvr{\varrho}
\bmdefine\BGs{\sigma}
\bmdefine\BGvs{\varsigma}
\bmdefine\BGt{\tau}
\bmdefine\BGj{\tau}
\bmdefine\BGu{\upsilon}
\bmdefine\BGo{\omega}
\bmdefine\BGx{\xi}
\bmdefine\BGy{\psi}
\bmdefine\BGz{\zeta}
\bmdefine\BGD{\Delta}
\bmdefine\BGF{\Phi}
\bmdefine\BGG{\Gamma}
\bmdefine\BGL{\Lambda}
\bmdefine\BGP{\Pi}
\bmdefine\BGT{\Theta}
\bmdefine\BGS{\Sigma}
\bmdefine\BGU{\Upsilon}
\bmdefine\BGO{\Omega}
\bmdefine\BGX{\Xi}
\bmdefine\BGY{\Psi}

\bmdefine\BFM{\mathfrak{M}}
\bmdefine\BFb{\mathfrak{b}}
\bmdefine\BFk{\mathfrak{k}}
\bmdefine\BFm{\mathfrak{m}}
\bmdefine\BFu{\mathfrak{u}}
\bmdefine\BFv{\mathfrak{v}}

\newcommand{\CA}{{\mathcal A}}

\newcommand{\CF}{{\mathcal F}}

\newcommand{\CH}{{\mathcal H}}

\newcommand{\CK}{{\mathcal K}}

\newcommand{\CR}{{\mathcal R}}

\bmdefine\BCA{{\mathcal A}}
\bmdefine\BCB{{\mathcal B}}
\bmdefine\BCC{{\mathcal C}}
\bmdefine\BCD{{\mathcal D}}
\bmdefine\BCE{{\mathcal E}}
\bmdefine\BCF{{\mathcal F}}
\bmdefine\BCG{{\mathcal G}}
\bmdefine\BCH{{\mathcal H}}
\bmdefine\BCI{{\mathcal I}}
\bmdefine\BCJ{{\mathcal J}}
\bmdefine\BCK{{\mathcal K}}
\bmdefine\BCL{{\mathcal L}}
\bmdefine\BCM{{\mathcal M}}
\bmdefine\BCN{{\mathcal N}}
\bmdefine\BCO{{\mathcal O}}
\bmdefine\BCP{{\mathcal P}}
\bmdefine\BCQ{{\mathcal Q}}
\bmdefine\BCR{{\mathcal R}}
\bmdefine\BCS{{\mathcal S}}
\bmdefine\BCT{{\mathcal T}}
\bmdefine\BCU{{\mathcal U}}
\bmdefine\BCV{{\mathcal V}}
\bmdefine\BCW{{\mathcal W}}
\bmdefine\BCX{{\mathcal X}}
\bmdefine\BCY{{\mathcal Y}}
\bmdefine\BCZ{{\mathcal Z}}

\bmdefine\Bzr{ 0}
\bmdefine\Ba{ a}
\bmdefine\Bb{ b}
\bmdefine\Bc{ c}
\bmdefine\Bd{ d}
\bmdefine\Be{ e}
\bmdefine\Bf{ f}
\bmdefine\Bg{ g}
\bmdefine\Bh{ h}
\bmdefine\Bi{ i}
\bmdefine\Bj{ j}
\bmdefine\Bk{ k}
\bmdefine\Bl{ l}
\bmdefine\Bm{ m}
\bmdefine\Bn{ n}
\bmdefine\Bo{ o}
\bmdefine\Bp{ p}
\bmdefine\Bq{ q}
\bmdefine\Br{ r}
\bmdefine\Bs{ s}
\bmdefine\Bt{ t}
\bmdefine\Bu{ u}
\bmdefine\Bv{ v}
\bmdefine\Bw{ w}
\bmdefine\Bx{ x}
\bmdefine\By{ y}
\bmdefine\Bz{ z}
\bmdefine\BA{ A}
\bmdefine\BB{ B}
\bmdefine\BC{ C}
\bmdefine\BD{ D}
\bmdefine\BE{ E}
\bmdefine\BF{ F}
\bmdefine\BG{ G}
\bmdefine\BH{ H}
\bmdefine\BI{ I}
\bmdefine\BJ{ J}
\bmdefine\BK{ K}
\bmdefine\BL{ L}
\bmdefine\BM{ M}
\bmdefine\BN{ N}
\bmdefine\BO{ O}
\bmdefine\BP{ P}
\bmdefine\BQ{ Q}
\bmdefine\BR{ R}
\bmdefine\BS{ S}
\bmdefine\BT{ T}
\bmdefine\BU{ U}
\bmdefine\BV{ V}
\bmdefine\BW{ W}
\bmdefine\BX{ X}
\bmdefine\BY{ Y}
\bmdefine\BZ{ Z}

\title{On feasibility of extrapolation of completely monotone functions}
\author{Henry J. Brown \and Yury Grabovsky}
\begin{document}
\maketitle
\begin{abstract}
  The feasibility of extrapolation of completely monotone functions can be
  quantified by examining the worst case scenario, whereby a pair of
  completely monotone functions agree on a given interval to a given relative
  precision, but differ as much as it is theoretically possible at a given
  point. We show that extrapolation is impossible to the left of the interval,
  while the maximal discrepancy to the right exhibits a power law typical for
  extrapolation of similar classes of complex analytic functions. The power
  law exponent is derived explicitly, and shows a precipitous drop immediately
  beyond the right end-point, with a subsequent decay to zero inversely
  proportional to the distance from the interval. The local extrapolation
  problem, where the worst discrepancy from a given completely monotone
  function is sought, is also analyzed. In this case explicit and easily
  verifiable optimality conditions are derived, enabling us to solve the
  problem exactly for a single decaying exponential. In the general case, our
  approach leads to a natural algorithm for computing solutions to the local
  extrapolation problem numerically. The methods developed in this paper can
  easily be adapted to other classes of analytic functions represented as integral
  transforms of positive measures with analytic kernels.
\end{abstract}

\tableofcontents


\section{Introduction}
\setcounter{equation}{0}
\label{sec:intro}
Theory of completely monotone functions (CMF) was developed in the 1920s and
1930s in the works of S. Bernstein \cite{bern29}, F. Hausdorff
\cite{hausdorff21}, V. Widder \cite{widd31,widd34} and Feller \cite{fell39} in
connection with the Markov moment problem \cite{Krein:1977:MMP}.  This class
of functions arises in several areas of mathematics
\cite{kimb74,bazh15,loan19,zast19} and remains of current research interest
(see reviews \cite{misa01,merk14}). Its importance in applications is rapidly
becoming more and more appreciated. Multiexponential models, whereby a
quantity of interest is a linear combination of decaying exponentials
\emph{with positive coefficients} are abundant in physics \cite{isvy99,gmm16},
engineering \cite{gpch00,rate13}, medicine \cite{rlfs09,dgvd10,dhjg13}, and
industry \cite{svhf02,pesc10}.

While the problem of central practical importance in applications is the
estimation of parameters of a multiexponential model
\cite{nies77,ener97,pesc10,nfjk19}, our goal is a theoretical analysis of reliability
of such procedures. To quantify the feasibility of recovery of such
functions from noisy measurements, we look for a pair of
completely monotone functions with relative discrepancy $\Ge$ on $[a,b]\subset[0,\infty)$, as
measured by the $L^{2}$ norm, that differ as much as possible at a given point
$x_{0}\not\in(a,b)$. We show that the discrepancy can be made as large
as one wishes for $0\le x_{0}\le a$, while for $x_{0}\ge b$ the relative discrepancy
scales as $\Ge^{\Gg(x_{0})}$, where
\begin{equation}
  \label{powerlaw}
  \Gg(x_{0})=\frac{2}{\pi}\arcsin\left(\frac{b-a}{x_{0}-a}\right),\quad x_{0}\ge b. 
\end{equation}
An analogous problem has been considered for the class of Stieltjes functions
(see e.g., \cite{stilt1894,Krein:1977:MMP,krnu98}) in \cite{grho-CEMP}.

Our general methodology, developed in \cite{grho-annulus,grho-gen,grho-CEMP}
for the Stieltjes class, can be applicable to many different classes of
functions that can be represented by integral transforms of positive measures
with analytic kernels. For example, CMFs are the Laplace transforms of
positive measures, while the Stieltjes functions, for which this approach was
first developed, are the Stieltjes transforms of positive measures
\cite{widd38}. The main technical difficulty is to link the problem of the
worst discrepancy between a pair of functions in our function class to the
much better understood problem of largest deviation from 0 among functions in
a reproducing kernel Hilbert space of analytic functions (such as Hardy
spaces) that are small on a curve in their domain of analyticity
\cite{ciulli69,mill70,fran90,vese99,deto18,trefe19}. The latter problem can be
reduced to the analysis of the asymptotics of eigenvalues and eigenfunctions
of specific integral operators
\cite{parfenov,gps03,pute17,grho-annulus,grho-gen}. The former is treated
using the same methodology as in \cite{grho-CEMP}, where a family of Hilbert
space norms was constructed that bridge the gap between the Hardy space norm
and the $L^{2}$ norm on the given curve.

We also investigate the local problem of finding a completely monotone
function $g(x)$, such that $\|f_{0}-g\|_{L^{2}(a,b)}\le\Ge$, that
maximizes and minimizes $f_{0}(x)-g(x)$, $x\not\in[a,b]$, where
$f_{0}(x)$ is a given completely monotone function, normalized by
$\|f_{0}\|_{L^{2}(a,b)}=1$. For this problem, we derive necessary and
sufficient conditions for the extremals $g(x)$, using the direct
analysis of the variation due to Caprini
\cite{capr74,capr80,capr81}. Caprini's method has the advantage of
suggesting an algorithm for computing the extremals numerically. The
implementation of this algorithm suggested the exact solutions for
$f_{0}(x)=e^{-x}$, which are then explicitly exhibited and analyzed.
The Caprini analysis-based approach has already been exploited in the
context of extrapolation of Stieltjes functions
\cite{grab_Stielt}. The details and implementation of an analogous
algorithm for completely monotone functions will be addressed
elsewhere.

There are three main innovations in this paper. The reduction to an integral
equation is now done using a new version of Kuhn-Tucker theorem, valid in all
locally convex topological vector spaces, making it applicable to a broader
class of problems. In the case under study, the resulting
integral operator has already been fully analyzed in \cite{kato16}. The theory
in \cite{grho-annulus,grho-gen} shows how the asymptotic behavior of
eigenfunctions for large eigenvalues leads to explicit formulas for exact
exponents in the power laws, like (\ref{powerlaw}).

The second innovation is a nontrivial construction of a continuous family of
Hilbert space norms that bridge the gap between the Hardy space norm and the
$L^{2}(a,b)$ norm. While the constructed family of norms does not bidge the
gap completely, it does so asymptotically. The explicit form of the power law
(\ref{powerlaw}) and the explicit asymptotics of the solution to the integral
equation are essential to establishing the link.

The third, is the worst case analysis of the local problem. There, the necessary
and sufficient conditions for extremality are found and used to compute the
two completely monotone functions deviating the most from a single decaying
exponential, with which they agree up to a relative precision $\Ge$ on a finite
interval. 


\section{Preliminaries and problem formulation}
\setcounter{equation}{0} 
\label{sec:prelim}
We say that $f:(0,\infty)\to[0,\infty)$ beongs to the class CM\footnote{The
  original definition of CMF is a nonnegative $C^{\infty}$ function on
  $(0,\infty)$, whose $k$th derivative is either always positive or always
  negative, depending on whether $k$ is even or odd. It was shown by
  S. Bernstein \cite{bern29} that the two definitions are equivalent.}, if it
can be represented as
\begin{equation}
  \label{intrep}
  f(x)=f_{\Gs}(x)=\int_{0}^{\infty}e^{-xt}d\Gs(t),
\end{equation}
where $\Gs$ is a positive, Borel-regular measure on $[0,\infty)$, such that
$f(x)<\infty$ for all $x>0$. In what follows, we will adopt the notation
$f_{\Gs}(x)$ to denote the function given by (\ref{intrep}).  Formula
(\ref{intrep}) implies that $f\in\CH(\CR)$, where
$\CR=\{z\in\bb{C}:\Re\,z>0\}$ is the complex right half-plane, and $\CH(\GO)$
denotes the space of all complex analytic functions on the open set
$\GO\subset\bb{C}$. The uniqueness property of analytic functions suggests
that the knowledge of a CMF on an interval $[a,b]$ should determine
such a function uniquely. In practice, where $f(x)$ is known only
approximately, the feasibility of extrapolation becomes a nontrivial question
that we address in this paper. Specifically, we assume that we know the values
of a CMF $f(x)$ on the interval $[a,b]$ up to a given relative
precision $\Ge$ in $L^{2}(a,b)$. We want to know how accurately we can
extrapolate this function outside of $[a,b]$. One immediately observes that
for any given CMF $f(x)$ the function
$f_{K}(x)=f(x)+\Ge\sqrt{2K}e^{-K(x-a)}$ is completely monotone for any $K>0$, and that
$\|f_{K}(x)-f(x)\|_{L^{2}(a,b)}\le\Ge$. However, for any $c\in[0,a]$, we can
make $f_{K}(c)-f(c)$ as large as we wish by choosing $K$ sufficiently
large. This shows that if we know that a pair of CMFs has a relative
discrepancy $\Ge$ in $L^{2}(a,b)$, their discrepancy at $x\le a$ can be made as
large as one wishes. We therefore conclude that we may assume, \WLOG, that $a=0$ and
rescale $b$ to 1. For this reason, we restrict our attention to a
subclass $\mathfrak{C}_{2}$ of CMFs defined by
\begin{equation}
  \label{C2def}
  \mathfrak{C}_{2}=\{f\in{\rm CM}:\|f\|_{2}<+\infty\},
\end{equation}
where $\|\cdot\|_{2}$ denotes the $L^{2}(0,1)$ norm.
We note that $\mathfrak{C}_{2}$ is not a vector space, but a convex cone. The
natural vector space the cone $\mathfrak{C}_{2}$ lies in is
$\mathfrak{X}=\mathfrak{C}_{2}-\mathfrak{C}_{2}$, which is a real vector
space, even though its elements are complex-analytic functions on $\CR$.

To formulate the problem of the worst case extrapolation, we
 denote
\begin{equation}
  \label{GDfg}
  \GD[f,g](x)=\frac{f(x)-g(x)}{\|f\|_{2}+\|g\|_{2}}, 
\end{equation}
describing the relative discrepancy at the point $x$ between the two functions
$\{f,g\}\subset\mathfrak{C}_{2}$. The worst case extrapolation problem is
\begin{equation}
  \label{fgproblem}
\GD^{x_{0}}(\Ge)=\max_{\|\GD[f,g]\|_{2}\le\Ge}|\GD[f,g](x_{0})|,
\end{equation}
where $x_{0}\ge 1$ is a given point. In other words, we seek the largest relative
discrepancy between two $\mathfrak{C}_{2}$ functions, that are at most $\Ge$
apart on $[0,1]$ in the $L^{2}$ sense. Our primary goal is to prove formula
(\ref{powerlaw}), which is equivalent to the following theorem.
\begin{theorem}
  \label{th:main}
Let $x_{0}\ge 1$, then
\begin{equation}
  \label{pl0}
\Gg(x_{0})\defeq\lim_{\Ge\to 0^{+}}\frac{\ln \GD^{x_{0}}(\Ge)}{\ln\Ge}=\frac{2}{\pi}\arcsin\left(\nth{x_{0}}\right),
\end{equation}
where $\GD^{x_{0}}(\Ge)$ is given by (\ref{fgproblem}), and the limit in
(\ref{pl0}) exists.
\end{theorem}
The idea of the proof is to relate (\ref{fgproblem}), that
we call the $(f,g)$-problem, to a simpler problem that we
know how to solve explicitly:
\begin{equation}
  \label{phiproblem}
  \GD^{x_{0}}_{*}(\Ge)=\max_{\phi\in\CA_{\Ge}}\phi(x_{0}),\quad
\CA_{\Ge}=\{\phi\in H:\|\phi\|\le 1,\|\phi\|_{2}\le\Ge\},
\end{equation}
where $H=\{\phi\in H^{2}(\CR):\bra{\phi(z)}=\phi(\bar{z})\}$
is a real subspace of the standard Hardy Hilbert space $H^{2}(\CR)$, and
where $\|\cdot\|$ is the multiple of the standard Hardy space norm
\begin{equation}
  \label{norm}
\|\phi\|^{2}=\sup_{x>0}\nth{2\pi}\int_{\bb{R}}|\phi(x+iy)|^{2}dy=\nth{2\pi}\int_{\bb{R}}|\phi(iy)|^{2}dy=\nth{\pi}\int_{0}^{\infty}|\phi(iy)|^{2}dy.  
\end{equation}
We call (\ref{phiproblem}) the $\phi$-problem. We note that the
Hardy space $H^{2}(\CR)$ is a reproducing kernel Hilbert space (see,
e.g. \cite{davis52}), and problems like (\ref{phiproblem}) have been
well-understood \cite{grho-annulus,grho-CEMP}. Our goal is to
show both that
\begin{equation}
  \label{goal}
    \Gg_{*}(x_{0})\defeq\lim_{\Ge\to 0}\frac{\ln\GD^{x_{0}}_{*}(\Ge)}{\ln\Ge}=\Gg(x_{0}),
\end{equation}
and that $\Gg_{*}(x_{0})$ is equal to the \rhs\ of (\ref{pl0}).
We follow here the same strategy that was used in \cite{grho-CEMP} in an analogous problem about
Stieltjes functions. The main difference (and
therefore difficulty) is that the Hardy norm $\|\cdot\|$ is not equivalent
to $\|\cdot\|_{2}$ on the convex cone $\mathfrak{C}_{2}$. This makes the direct
comparison between $\Gg(x_{0})$ and $\Gg_{*}(x_{0})$ impossible.

Our way of resolving this difficulty is to bridge the gap between the two
norms by introducing a continuous family of intermediate Hardy space-like
norms of increasing strength on
$\mathfrak{X}=\mathfrak{C}_{2}-\mathfrak{C}_{2}$, all of which are a
equivalent to $\|\cdot\|_{2}$ on $\mathfrak{C}_{2}$. Each norm in the family gives rise to
the corresponding $\phi$-problem (\ref{phiproblem}), where it replaces the
Hardy norm $\|\cdot\|$. What permits us to close the circle of
inequalities between the corresponding power law exponents $\Gg$ is our
ability to solve the the original $\phi$-problem (\ref{phiproblem}) explicitly and thus
estimate all of its intermediate norms directly. We remark that it is the absence of
the explicit solution of the $\phi$-problem in \cite{grho-CEMP} that prevented
us from completing the rigorous proof of the analog of (\ref{goal}) in the
context of Stieltjes functions. 


\section{Existence of maximizers}
\setcounter{equation}{0}
The goal of this section is to prove the attainment of the maxima both in
(\ref{fgproblem}) and in (\ref{phiproblem}).
We start by proving the representation property of functions in $H$.
\begin{lemma}
  \label{lem:Hrep}
For any $\phi\in H$, there exists $\Gs\in L^{2}(0,\infty)$, $\Gs(t)\in\bb{R}$,
such that
\begin{equation}
  \label{Hfrep}
  \phi(z)=\int_{0}^{\infty}\Gs(t)e^{-zt}dt,\quad\Re\,z>0.
\end{equation}
\end{lemma}
\begin{proof}
If $\phi\in H$, then $\phi(iy)\in L^{2}(\bb{R})$, and therefore, there
exists $\Gs\in L^{2}(\bb{R})$, such that
\[
\phi(iy)=\hat{\Gs}(y)=\int_{\bb{R}}\Gs(t)e^{-iyt}dt.
\]
The symmetry of functions in $H$, i.e. $\bra{\phi(iy)}=\phi(-iy)$ implies that
$\Gs(t)\in\bb{R}$. Since $H$ is a subspace of the Hardy space $H^{2}(\CR)$,
for any $\phi\in H$ there is the Kramers-Kronig relation
\cite{kron26,kramers27} that says that the real part of $\phi(iy)$ is the
Hilbert transform of its imaginary part. Since the Hilbert transform is a
Fourier multiplier operator by $i\,$sign$(t)$, the Kramers-Kronig relation can be
written as $\Re\hat{g}(y)=0$, where $g(t)=\Gs(t)-\Gs(t)$sign$(t)$. But then,
$g(t)$ has to be an odd function on $\bb{R}$. We conclude that
$g(t)$ must be identically zero since it is zero on $(0,\infty)$. It follows
that $\Gs(t)=0$ for all $t<0$, and
\[
\phi(iy)=\hat{\Gs}(y)=\int_{0}^{\infty}\Gs(t)e^{-iyt}dt,\quad y\in\bb{R}.
\]
Therefore representation (\ref{Hfrep}) holds since
Hardy functions possess a unique analytic extension into the complex right half-plane. 
\end{proof}
We remark that in view of representation (\ref{Hfrep}) the Hardy inner product
in $H$ can also be computed as
\begin{equation}
  \label{Hinpr}
  (\phi_{\Gs},\phi_{\mu})=(\Gs,\mu)_{L^{2}(0,\infty)}.
\end{equation}
To establish attainment in (\ref{phiproblem}), we need the following lemma.
\begin{lemma}
  \label{lem:HL2}
For any $\phi\in H$
  \begin{equation}
    \label{HL2}
    \|\phi\|_{2}\le\sqrt{\pi}\|\phi\|.
  \end{equation}
\end{lemma}
\begin{proof}
Using representation (\ref{Hfrep}), we have
\[
\|\phi\|_{2}^{2}\le\|\phi\|_{L^{2}(0,\infty)}^{2}=\int_{0}^{\infty}\int_{0}^{\infty}\frac{\Gs(s)\Gs(t)}{s+t}dsdt=
\pi ((H\Gs)(-t),\Gs(t))_{L^{2}(\bb{R})},
\]
where $H\Gs$ is the Hilbert transform and $\Gs\in L^{2}(0,\infty)$ is extended by
zero on $(-\infty,0)$ to yield a function in $L^{2}(\bb{R})$.
Hence
\[
\|\phi\|_{2}^{2}\le\pi\|(H\Gs)(-t)\|_{L^{2}(\bb{R})}\|\Gs\|_{L^{2}(\bb{R})}=\pi\|\Gs\|_{L^{2}(\bb{R})}^{2}=\pi\|\phi\|^{2}.
\]
\end{proof}
The attainment in (\ref{phiproblem}) is now obvious since $\CA_{\Ge}$ is
closed, convex, and bounded in $H$, and the evaluation functional
$H\ni\phi\mapsto\phi(x_{0})$ is continuous. (It is obvious, for example, from
representation (\ref{Hfrep}) and the fact that $e^{-x_{0}t}\in
L^{2}(0,\infty)$.)

To prove the attainment in (\ref{fgproblem}), we need the following lemma.
\begin{lemma}
  \label{lem:L2star}
For any $f\in\mathfrak{C}_{2}$
  \begin{equation}
    \label{L2star}
    \|f_{\Gs}\|_{2}\ge\|\Gs\|_{*},
  \end{equation}
where
\begin{equation}
  \label{star}
  \|\Gs\|_{*}=\int_{0}^{\infty}\frac{d\Gs(t)}{t+1}.
\end{equation}
\end{lemma}
\begin{proof}
  Using representation (\ref{intrep}), we compute
\[
\|f_{\Gs}\|_{2}^{2}=\int_{0}^{\infty}\int_{0}^{\infty}\frac{1-e^{-(t+s)}}{t+s}d\Gs(t)d\Gs(s).
\]
Now observe that since $\min_{x\ge 0}x^{-1}(x+1)(1-e^{-x})=1$, then
for any $s>0$ and $t>0$ we have
\[
  \frac{1-e^{-(t+s)}}{t+s}\ge\frac{1}{t+s+1}\ge\nth{(t+1)(s+1)}.
\]
Inequality (\ref{L2star}) follows.
\end{proof}
We are now ready to prove the attainment in (\ref{fgproblem}). 
\begin{theorem}
  The maximum in (\ref{fgproblem}) is attained.
\end{theorem}
\begin{proof}
Let $\{f_{n},g_{n}\}\subset\mathfrak{C}_{2}$ be a minimizing sequence for
(\ref{fgproblem}). Then sequences
\[
\Tld{f}_{n}=\frac{f_{n}}{\|f_{n}\|_{2}+\|g_{n}\|_{2}},\quad
\Tld{g}_{n}=\frac{g_{n}}{\|f_{n}\|_{2}+\|g_{n}\|_{2}}
\]
are bounded in $L^{2}(0,1)$. By Lemma~\ref{lem:L2star} the corresponding
measures $\Tld{\Gs}_{n}$, $\Tld{\mu}_{n}$ are bounded in $X^{*}$, where
\begin{equation}
  \label{X}
  X=\left\{\Phi\in C([0,\infty)):\lim_{t\to\infty}(1+t)\Phi(t)=0\right\}
\end{equation}
is a Banach space with $\|\Phi\|_{X}=\sup_{t>0}(t+1)|\Phi(t)|$.
Since $X$ is separable, there are weak-* converging
subsequences, not relabeled, $\Tld{\Gs}_{n}\wk{*}\Gs_{0}\in X^{*}$,
$\Tld{\mu}_{n}\wk{*}\mu_{0}\in X^{*}$. Since $e^{-x_{0}t}\in X$, we conclude
that 
$\Tld{f}_{n}(x_{0})-\Tld{g}_{n}(x_{0})\to f_{0}(x_{0})-g_{0}(x_{0})$, where
\[
f_{0}(x)=\int_{0}^{\infty}e^{-xt}d\Gs_{0}(t),\quad g_{0}(x)=\int_{0}^{\infty}e^{-xt}d\mu_{0}(t).
\]
In fact, the pointwise convergence of $\Tld{f}_{n}$ and $\Tld{g}_{n}$ together
with their weak precompactness in $L^{2}(0,1)$ implies that $\Tld{f}_{n}\weak
f_{0}$, and $\Tld{g}_{n}\weak
g_{0}$ in $L^{2}(0,1)$. The weak lower semicontinuity of the norm in
$L^{2}(0,1)$ implies that $\{f_{0},g_{0}\}\subset\CA_{\Ge}$ and therefore
attain the maximum in (\ref{fgproblem}).
\end{proof}


\section{The $\phi$-problem}
\setcounter{equation}{0}
\label{sec:phi}
The goal of this section is to solve the $\phi$-problem (\ref{phiproblem}).

\subsection{Reduction to an integral equation}
\label{sub:intreduct}
The $\phi$-problem (\ref{phiproblem}) asks to maximize a linear continuous
functional on the Hilbert space $H$ over a convex and closed subset
$\CA_{\Ge}\subset H$. A new general version of the Kuhn-Tucker theorem, valid
in all locally convex topological vector spaces, is formulated and proved in
Appendix~\ref{app:KT}. In order to apply it, we need to describe the
admissible set of functions $\CA_{\Ge}$ in the standard form (\ref{Kdef}). To
do so, we first observe that
\[
\|\phi_{\Gs}\|=\|\Gs\|_{L^{2}(0,+\infty)}=\sup_{\|\Psi\|_{L^{2}(0,+\infty)}\le 1}\int_{0}^{\infty}\Psi(t)\Gs(t)dt,\qquad
\|\phi\|_{2}=\sup_{\|\psi\|_{2}\le 1}\int_{0}^{1}\phi(x)\psi(x)dx.
\]
Let us show that $L^{2}(0,1)$ acts on $H$ by weakly continuous
functionals, where the action of $\psi\in L^{2}(0,1)$ on $H$ is
defined by 
\[
\phi\mapsto(\phi,\psi)_{2}=\int_{0}^{1}\phi(x)\psi(x)dx.
\]
Indeed, $|(\phi,\psi)_{2}|\le\|\psi\|_{2}\|\phi\|_{2}\le
\sqrt{\pi}\|\psi\|_{2}\|\phi\|$, by Lemma~\ref{lem:HL2}. We also have
\[
(\phi_{\Gs},\psi)_{2}=\int_{0}^{1}\phi_{\Gs}(x)\psi(x)dx=\int_{0}^{\infty}(\GL\psi)(t)\Gs(t)dt,\qquad
(\GL\psi)(t)=\int_{0}^{1}\psi(x)e^{-xt}dx,
\]
while the bound
\[
  |(\phi_{\Gs},\psi)_{2}|\le\sqrt{\pi}\|\psi\|_{2}\|\phi_{\Gs}\|=
  \sqrt{\pi}\|\psi\|_{2}\|\Gs\|_{L^{2}(0,\infty)}
\]
implies
\[
\|\GL\psi\|_{L^{2}(0,\infty)}\le\sqrt{\pi}\|\psi\|_{2}.
\]
Thus, we obtain the desired description of $\CA_{\Ge}$
\[
\CA_{\Ge}=\{\phi_{\Gs}\in H:(\Gs,\Psi)_{L^{2}(0,\infty)}\le 1\ \forall\|\Psi\|_{L^{2}(0,\infty)}\le 1,\
(\Gs,\GL\psi)_{L^{2}(0,\infty)}\le\Ge\ \forall\|\psi\|_{2}\le 1\}.
\]
In order to apply the Kuhn-Tucker theorem, we need to compute the smallest closed convex cone
$\Hat{\CF}\subset H\times\bb{R}$ containing the set
\[
\CF=\{(\Psi,1):\|\Psi\|_{L^{2}(0,\infty)}\le 1\}\cup\{(\GL\psi,\Ge)_{L^{2}(0,\infty)}:\|\psi\|_{2}\le 1\}.
\]
We can characterize $\Hat{\CF}$ as
\[
\Hat{\CF}=\{(\Psi+\GL\psi,A+\Ge B):\|\Psi\|_{L^{2}(0,\infty)}\le A,\
\|\psi\|_{2}\le B,\ A\ge 0,\ B\ge 0\}.
\]
Indeed, it is obvious both that $\Hat{\CF}$ is a convex cone and
that each element of $\Hat{\CF}$ is a nonnegative linear combination of two
elements from $\CF$.
To prove that $\Hat{\CF}$ is closed suppose that 
\[
\Psi_{n}+\GL\psi_{n}\to P\ \text{in }L^{2}(0,\infty),\quad A_{n}+\Ge B_{n}\to\Ga,\quad \|\Psi_{n}\|_{L^{2}(0,\infty)}\le A_{n},\quad
\|\psi_{n}\|_{2}\le B_{n}.
\]
Then $A_{n}\le A_{n}+\Ge B_{n}$ and $B_{n}\le(A_{n}+\Ge B_{n})/\Ge$. Hence, we
can extract convergent subsequences (not relabeled) of $A_{n}\to A$ and
$B_{n}\to B$. We can also extract the weakly convergent subsequences (not
relabeled) $\Psi_{n}\weak\Psi$, $\psi_{n}\weak\psi$. The weak lower
semicontinuity of the norms implies that $\|\Psi\|_{L^{2}(0,\infty)}\le A$,
$\|\psi\|_{2}\le B$, while $A+\Ge B=\Ga$ and $\Psi+\GL\psi=P$. Thus,
$(P,\Ga)\in\Hat{\CF}$, and we conclude that $\Hat{\CF}$ is weakly closed. Now,
according to the Kuhn-Tucker theorem~\ref{th:KT},
\begin{equation}
  \label{KT}
\GD^{x_{0}}_{*}(\Ge)=\max_{\phi\in\CA_{\Ge}}\phi(x_{0})=\min_{\psi\in L^{2}(0,1)}\left(
\Ge\|\psi\|_{2}+\left\|\GL\psi-e^{-x_{0}t}\right\|_{L^{2}(0,\infty)}\right).
\end{equation}
The minimizer $\psi_{\Ge}$ in (\ref{KT}) exists for any fixed $\Ge>0$, because this is a convex
and coercive variational problem. However, this problem is difficult analyze; Hence, we are going to modify the
maximization problem (\ref{KT}) to make it more tractable, while using
our understanding of the relation between solutions of
(\ref{phiproblem}) and (\ref{KT}) to obtain the maximizer in
(\ref{phiproblem}).
Using that  for $1/p+1/q=1$,
\begin{equation}
  \label{pqineq}
\nth{p}\left(\frac{p}{q}a^{2}+b^{2}\right)=\frac{a^{2}}{q}+\frac{b^{2}}{p}\le(a+b)^{2}\le
pa^{2}+qb^{2}=q \left(\frac{p}{q}a^{2}+b^{2}\right),
\end{equation}
we conclude that for the sake of understanding the asymptotic behavior of
$\GD_*^{x_0}(\Ge)$, we can replace the variational problem (\ref{KT}) by a
quadratic one:
\begin{equation}
  \label{quadmin}
Q_{x_{0}}(\Gve)=\min_{\psi\in L^{2}(0,1)}\eps^{2}\|\psi\|_{2}^{2}+\left\|\GL\psi-e^{-x_{0}t}\right\|_{L^{2}(0,\infty)}^{2},
\end{equation}
where $\eps =\Ge\sqrt{p(\Ge)/q(\Ge)}$, and where the parameters $p(\Ge)$,
$q(\Ge)$, satisfying $1/p(\Ge)+1/q(\Ge)=1$ will be chosen later to optimize the
upper bound that, according to (\ref{pqineq}), reads
\begin{equation}
  \label{UB}
\GD^{x_{0}}_{*}(\Ge)^{2}\le q(\Ge)Q_{x_{0}}(\Gve),\qquad\eps =\Ge\sqrt{\frac{p(\Ge)}{q(\Ge)}}.
\end{equation}

The advantage of the quadratic minimization problem (\ref{quadmin}) over
(\ref{KT}) is that the minimizer $\psi_{\eps}$ of (\ref{quadmin}) solves a
linear integral equation
\begin{equation}
 \label{inteq}
 \eps^{2}\psi(x)+(K\psi)(x)=\frac{1}{x_{0}+x},\quad x\in[0,1], 
\end{equation}
where $K:L^{2}(0,1)\to L^{2}(0,1)$,
\[
  (K\psi)(x)=(\GL^{*}\GL\psi)(x)=\int_{0}^{1}\frac{\psi(y)dy}{x+y}
\]
is a bounded, nonnegative, and self-adjoint operator. Hence,
(\ref{inteq}) has a unique solution $\psi_{\Gve}\in L^{2}(0,1)$. 

Representing the kernel $(x+y)^{-1}$ of the integral operator in the form
\[
\nth{x+y}=\int_{0}^{\infty}e^{-xt}e^{-yt}dt,
\]
we conclude that the solution $\psi_{\Gve}$ of (\ref{inteq}) satisfies
\begin{equation}
  \label{phiepsrep}
  \psi_{\eps}(x)=\nth{\eps^{2}}\int_{0}^{\infty}(e^{-x_{0}t}-\GL\psi_{\eps})e^{-xt}dt.
\end{equation}
This shows that $\psi_{\Gve}\in L^{2}(0,1)$ has the unique extension, also
denoted $\psi_{\eps}\in H$, which has a representation (\ref{Hfrep}), with 
$\Gs=\Gve^{-2}(e^{-x_{0}t}-\GL\psi_{\eps})\in L^{2}(0,\infty)$. Therefore, in view of (\ref{Hinpr}), we have
\begin{equation}
  \label{Hnormrep}
  \|\psi_{\eps}\|=\nth{\eps^{2}}\|\GL\psi_{\eps}-e^{-x_{0}t}\|_{L^{2}(0,\infty)}.
\end{equation}

Setting $x=x_{0}$ in (\ref{phiepsrep}), we obtain
\[
\psi_{\eps}(x_{0})=\nth{\eps^{2}}\int_{0}^{\infty}(e^{-x_{0}t}-\GL\psi_{\eps})e^{-x_{0}t}dt.
\]
Multiplying (\ref{phiepsrep}) by $\psi_{\eps}$ and integrating over
$[0,1]$, we get
\[
\|\psi_{\eps}\|_{2}^{2}=\nth{\eps^{2}}\int_{0}^{\infty}(e^{-x_{0}t}-\GL\psi_{\eps})\GL\psi_{\eps}dt.
\]
Subtracting the two equations and taking (\ref{Hnormrep}) into account yields
\begin{equation}
  \label{Pythagoras}
\|\psi_{\eps}\|_{2}^{2}+\eps^{2} \|\psi_{\eps} \|^{2}=\psi_{\eps}(x_{0}).
\end{equation}
This relation implies that $Q_{x_0}(\Gve)=\eps^{2}\psi_{\eps}(x_{0})$, while the upper
bound (\ref{UB}) becomes
\begin{equation}
  \label{maxeval}
\GD_*^{x_0}(\Ge)^{2}\le q(\Ge)\eps^{2}\psi_{\eps}(x_{0})=\Ge^{2}p(\Ge)\psi_{\eps}(x_{0}).
\end{equation}
The lower bound for $\GD_*^{x_0}(\Ge)$ is obtained by using a test
function
\begin{equation}
  \label{phipsi}
  \phi_{\Ge}=\frac{\Ge \psi_{\eps}}{\|\psi_{\eps}\|_2}\in H,
\end{equation}
which obviously satisfies $\|\phi_{\Ge}\|_2=\Ge$, and where $p(\Ge)$ is chosen
so that $\|\phi_{\Ge}\|=1$. Specifically, using (\ref{Pythagoras}), we have
\[
\|\phi_{\Ge}\|^2 = \frac{\Ge^2\| \psi_{\eps}\|^{2}}{\|\psi_{\eps}\|_{2}^{2}}
=\frac{q(\Ge)}{p(\Ge)}\left(\frac{\psi_{\eps}(x_0)}{\|\psi_{\eps}\|_{2}^{2}}-1\right)=
\frac{\frac{\psi_{\eps}(x_0)}{\|\psi_{\eps}\|_{2}^{2}}-1}{p(\Ge)-1}.
\]
Setting $\|\phi_{\Ge}\|^{2}=1$, we obtain
\begin{equation}
  \label{pofeps}
  p(\Ge)=\frac{\psi_{\eps}(x_0)}{\|\psi_{\eps}\|_{2}^{2}}=1+\frac{\Gve^{2}\|\psi_{\Gve}\|^{2}}{\|\psi_{\Gve}\|_{2}^{2}}\in(1,+\infty),
\end{equation}
due to (\ref{Pythagoras}). The choice (\ref{pofeps}) of $p(\Ge)$ implies that
$\phi_{\Ge}\in\CA_{\Ge}$, yielding the lower bound for $\GD_*^{x_0}(\Ge)$
\[
(\GD_*^{x_0}(\Ge))^{2}\ge(\phi_{\Ge}(x_0))^2=
\frac{\Ge^{2}\psi_{\eps}(x_{0})^{2}}{\|\psi_{\eps}\|_2^{2}}=\Ge^{2}p(\Ge)\psi_{\eps}(x_{0}),
\]
provided $p(\Ge)$ is given by (\ref{pofeps}). Hence, the lower bound for
$\GD_*^{x_0}(\Ge)$ agrees with the upper bound (\ref{UB}), and therefore,
\begin{equation}
  \label{exactasymp}
 \GD_*^{x_0}(\Ge)=\frac{\Ge\psi_{\eps}(x_{0})}{\|\psi_{\eps}\|_2},
\end{equation}
where $\psi_{\eps}$ solves (\ref{inteq}) and $\Gve$ and $\Ge$ are related by
\begin{equation}
  \label{GeGve}
  \Ge=\frac{\|\psi_{\Gve}\|_{2}}{\|\psi_{\Gve}\|},
\end{equation}
which is easy to obtain combining (\ref{pofeps}) and the formula for $\Gve$
from (\ref{UB}). Substituting this into (\ref{exactasymp}), we also obtain
\begin{equation}
  \label{exactasymp0}
 \GD_*^{x_0}(\Ge)=\frac{\psi_{\eps}(x_{0})}{\|\psi_{\eps}\|}.
\end{equation}
We can use formulas (\ref{GeGve}) and (\ref{exactasymp0}) to establish the
explicit leading order asymptotics of $\GD_*^{x_0}(\Ge)$, if we can compute
the explicit leading order asymptotics of the \rhs s in (\ref{GeGve}) and
(\ref{exactasymp0}). Specifically, if $E_{0}(\Gve)$ and $E_{1}(\Gve)$ are
continuous and monotone increasing functions on $[0,1)$, such that
$E_{0}(0)=E_{1}(0)=0$, and
\[
\lim_{\Gve\to 0^{+}}\frac{\psi_{\eps}(x_{0})}{E_{0}(\Gve)\|\psi_{\eps}\|}=1,\qquad
\lim_{\Gve\to 0^{+}}\frac{\|\psi_{\Gve}\|_{2}}{E_{1}(\Gve)\|\psi_{\Gve}\|}=1,
\]
then we want to conclude that
\begin{equation}
  \label{exactasymp1}
  \lim_{\Ge\to 0^{+}}\frac{\GD_*^{x_0}(\Ge)}{E_{0}(E_{1}^{-1}(\Ge))}=1.
\end{equation}
Since $\Ge(\Gve)\sim E_{1}(\Gve)$, then the assumed properties of
$E_{1}(\Gve)$ imply that $\Ge\to 0^{+}$ \IFF $\Gve\to 0^{+}$. Then,
\[
\lim_{\Ge\to 0^{+}}\frac{\GD_*^{x_0}(\Ge)}{E_{0}(E_{1}^{-1}(\Ge))}=
\lim_{\Gve\to 0^{+}}\frac{E_{0}(\Gve)\frac{\psi_{\eps}(x_{0})}{E_{0}(\Gve)\|\psi_{\eps}\|}}
{E_{0}\left(E_{1}^{-1}\left(E_{1}(\Gve)\frac{\|\psi_{\Gve}\|_{2}}{E_{1}(\Gve)\|\psi_{\Gve}\|}\right)\right)}.
\]
Thus, (\ref{exactasymp1}) follows, if functions $E_{0}$ and $E_{1}$ have the
additional property
\begin{equation}
  \label{E0E1prop}
 \lim_{\Gve\to 0^{+}}\frac{E_{0}(\Gve)}{E_{0}(E_{1}^{-1}(E_{1}(\Gve)r(\Gve)))}=1,
\end{equation}
whenever $r(\Gve)\to 1$, as $\Gve\to 0^{+}$. It is not difficult to give an
example of continuous and monotone increasing functions $E_{0}$ and $E_{1}$,
with $E_{0}(0)=E_{1}(0)=0$ that fail to satisfy (\ref{E0E1prop}).

\subsection{Solution of the integral equation}
To solve the integral equation (\ref{inteq}), we diagonalize the
bounded self-adjoint operator $K$. The problem of computing the eigenfunctions of $K$ can be
related to a problem about the truncated Hilbert transform
\[
(H_{1}u)(\xi)=P.V.\int_{0}^{1}\frac{u(y)dy}{\xi-y},
\]
regarded as a map $H_{1}:L^{2}(0,1)\to L^{2}(-1,0)$, that has been solved in \cite{kato16}. 
The relation between the operators $K$ and $H_{1}$ is expressed by the
formula $K^{2}=H_{1}^{*}H_{1}$, which shows that if $u$ is an eigenfunction of $K$ with
eigenvalue $\nu>0$, then $u$ is also a singular function of $H_{1}$ with
singular value $\nu$. Conversely, if $u$ is a singular function of $H_{1}$ with
singular value $\nu$, then $K^{2}u=\nu^{2}u$, which implies that
$(K+\nu)(K-\nu)u=0$. Since $K$ is a bounded nonnegative operator, the
operator $K+\nu$ is invertible and we conclude that $Ku=\nu u$. In 
\cite{kato16} it was shown that the spectrum of $H_{1}^{*}H_{1}$ is
continuous and its eigenfunctions can be found explicitly by observing that
the differential operator
\[
  (Lu)(x)=-(x^{2}(1-x^{2})u'(x))'+2x^{2}u(x)
\]
commutes with $H_{1}^{*}H_{1}$. We can easily verify that $L$ also
commutes with $K$. That means that if $u$ is an eigenfunction of $L$
corresponding to the eigenvalue $\Gl$, then $\Gl Ku=KLu=LKu$. Hence,
$Ku$ is also an eigenfunction of $L$ with the eigenvalue $\Gl$. As
computed in \cite{kato16}, the
eigenspaces of $L$ are all one-dimensional, spanned by 
\begin{equation}
  \label{uofx}
  u(x;\mu)=x^{-\hf+i\mu}F\left(\left[\nth{4}+\frac{i\mu}{2},\frac{3}{4}+\frac{i\mu}{2}\right],
[1];1-x^{2}\right),\quad\Gl=\mu^{2}+\nth{4},\ \mu\ge 0,
\end{equation}
where $F([a,b],[c];z)$ is the Gauss hypergeometric function. We conclude
that functions $u(x;\mu)$ are eigenfunctions of $K$. The corresponding
eigenvalues, are the singular values of $H_{1}$, which, according to
\cite{kato16}, are given by
\begin{equation}
  \label{numu}
  \nu(\mu)=\frac{\pi}{\cosh(\pi\mu)}. 
\end{equation}
We note that the function $z\mapsto F([a,b],[c];z)$ is analytic in
$\bb{C}\setminus[1,+\infty)$. Therefore, $u(x;\mu)$, given by (\ref{uofx}), is analytic in the
complex right half-plane.  
The orthogonality of the eigenfunctions is conveniently expressed
in terms of the ``$u$-transform'' and its inverse (see \cite{kato16}):
\begin{equation}
  \label{utransform}
  \hat{f}(\mu)=\int_{0}^{1}f(x)u(x;\mu)dx,\qquad
f(x)=\int_{0}^{\infty}\hat{f}(\mu)u(x;\mu)\mu\tanh(\pi\mu)d\mu.
\end{equation}
Multiplying the second equation by $f(x)$ and integrating gives the
generalized Plancherel formula
\begin{equation}
  \label{unorm}
  \|f\|_{2}^{2}=\int_{0}^{\infty}|\hat{f}(\mu)|^{2}\mu\tanh(\pi\mu)d\mu.
\end{equation}

The knowledge of the eigenfunctions of $K$ permits us to solve the integral equation (\ref{inteq}):
\begin{equation}
  \label{phirep}
  \psi_{\Gve}(x)=\int_{0}^{\infty}\frac{u(x_{0};\mu)u(x;\mu)\mu\tanh(\pi\mu)}{2\hat{\Gve}^{2}\cosh(\pi\mu)+1}d\mu,\quad\hat{\Gve}=\frac{\Gve}{\sqrt{2\pi}}.
\end{equation}
Moreover,
\begin{equation}
  \label{phinorm}
  \|\psi_{\Gve}\|_{2}^{2}=\int_{0}^{\infty}\frac{u(x_{0};\mu)^{2}\mu\tanh(\pi\mu)}{(2\hat{\Gve}^{2}\cosh(\pi\mu)+1)^{2}}d\mu, 
\end{equation}
while
\begin{equation}
  \label{phix0}
  \psi_{\Gve}(x_{0})=\int_{0}^{\infty}\frac{u(x_{0};\mu)^{2}\mu\tanh(\pi\mu)}{2\hat{\Gve}^{2}\cosh(\pi\mu)+1}d\mu.
\end{equation}
Substituting (\ref{phinorm}) and (\ref{phix0}) into (\ref{Pythagoras}) gives
\begin{equation}
  \label{phiHnorm}
  \|\psi_{\Gve}\|^{2}=\nth{\pi}\int_{0}^{\infty} \frac{u(x_{0};\mu)^{2}\mu\sinh(\pi \mu)}{(2\hat{\Gve}^{2}\cosh(\pi\mu)+1)^{2}} d\mu. 
\end{equation}
When $x=x_{0}>1$ the coefficient $-x^{2}(1-x^{2})$ in the differential
operator $L$ becomes positive, and we expect the
eigenfunctions $u(x_{0};\mu)$ to grow
exponentially as $\mu\to\infty$. Thus, if we set $\Gve=0$ in (\ref{phinorm})
and (\ref{phix0}), we obtain exponentially divergent integrals, while they
remain convergent for each $\Gve>0$. Thus, $\|\psi_{\Gve}\|_{2}\to\infty$ and
$\psi_{\Gve}(x_{0})\to\infty$, as $\Gve\to 0$, and the precise asymptotics of these
quantities, as $\Gve\to 0$, would depend on the rate of exponential growth of
$u(x_{0};\mu)$, as $\mu\to\infty$.

\subsection{Asymptotics of $\GD_*^{x_0}(\Ge)$}
In this section the notation $A(\Ge)\sim B(\Ge)$ means  $A(\Ge)/B(\Ge)\to 1$,
as $\Ge\to 0^{+}$. Similarly, $A(\mu)\sim B(\mu)$ means  $A(\mu)/B(\mu)\to 1$,
as $\mu\to+\infty$. The goal of this section is to compute the following
explicit asymptotics\footnote{For our purposes, we only need the exponent. We
  derive the explicit formula for $C_{*}(x_{0})$ because we can, and because
  the technique we use may be of broader interest.} of $\GD_{*}^{x_{0}}(\Ge)$.
\begin{theorem}
  \label{th:GDstar}
  \begin{equation}
    \label{DGstar}
    \GD_{*}^{x_{0}}(\Ge)\sim
    \begin{cases}
      C_{*}(x_{0})\Ge^{\frac{2}{\pi}\arcsin\left(\frac{1}{x_{0}}\right)},&x_{0}>1,\\
      \frac{\sqrt{2}}{\pi}\Ge|\ln\Ge|,&x_{0}=1,
    \end{cases}
\end{equation}
where
\begin{equation}
  \label{Cstar}
  C_{*}(x_{0})=\hf\sqrt{\frac{x_{0}}{2(x_{0}^{2}-1)\arcsin\left(\frac{1}{x_{0}}\right)}}
\left(\frac{2\pi\arcsin\left(\frac{1}{x_{0}}\right)}{\arccos\left(\frac{1}{x_{0}}\right)}\right)^{\frac{\arccos\left(\frac{1}{x_{0}}\right)}{\pi}}.
  \end{equation}
\end{theorem}
Formula (\ref{exactasymp1}) expresses the asymptotics of $\GD_*^{x_0}(\Ge)$ in
terms of the asymptotics of $\|\psi_{\Gve}\|_{2}$, $\psi_{\Gve}(x_0)$, and
$\|\psi_{\Gve}\|$, given by (\ref{phinorm}), (\ref{phix0}), and
(\ref{phiHnorm}), respectively. In turn, these depend on the asymptotics of
$u(x_{0};\mu)$, as $\mu\to\infty$. The following lemma gives the asymptotics
of $u(z;\mu)$, as $\mu\to\infty$ for all $z$ in the complex right half-plane,
excluding the interval $[0,1]$. While in this section we will only need the
asymptotics of $u(z;\mu)$ for real $z>1$, the asymptotics for other values of
$z$ will also be required later on.
\begin{lemma}
  \label{lem:ux0asym}
Let $u(x;\mu)$ be the eigenfunctions of the integral operator $K$. Then formula
(\ref{uofx}) gives the analytic extension of $u(x;\mu)$ from $[0,1]$ to the complex
right half-plane. Moreover,
\begin{equation}
  \label{uofx0}
u(z;\mu)\sim
R(z)\frac{e^{\mu\Ga(z)}}{\sqrt{2\pi\mu}},\text{ as }\mu\to\infty,
\end{equation}
for every
$z\in\GO=\{z\in\bb{C}:\Re z>0,\ z\not\in[0,1]\}$, where
\begin{equation}
  \label{Ralpha}
  R(z)=z^{-1/2}(z^{2}-1)^{-1/4},\qquad\Ga(z)=\arccos\left(\nth{z}\right)
  =i\ln\left(\frac{1-i\sqrt{z^{2}-1}}{z}\right).
\end{equation}
and where the principal branches of the natural logarithm and all
fractional powers are used.
\end{lemma}
The proof is a straightforward application of the asymptotic formulas
for the Gauss hypergeometric function from \cite{khda14}. The required
calculations needed to apply these formulas to our specific case are
detailed in Appendix~\ref{app:HG}. We also remark $u(1;\mu)=1$ and
that the asymptotics of $u(x;\mu)$ for $x\in[0,1)$ is given
in \cite[formula (4.34)]{kato16}.

The exponential growth of $u(z;\mu)$ as $\mu\to\infty$, described by
Lemma~\ref{lem:ux0asym} permits us to compute the explicit asymptotics
of $\psi_{\Gve}(z)$, $\|\psi_{\eps}\|_{2}$ and $\|\psi_{\eps}\|$,
given by (\ref{phirep}), (\ref{phinorm}), and (\ref{phiHnorm}),
respectively. This is made possible by the following lemma.
\begin{lemma}
\label{lem:intasymp}
  Suppose that $v\in C([0,\infty))$ is such
  that $v(\mu)\to 1$, as $\mu\to+\infty$, $k\in\{1,2\}$, and
  $\Re\Gb\in(0,k)$. Then
  \begin{equation}
    \label{intasymp}
    \int_{0}^{\infty}\frac{e^{\pi\Gb\mu}v(\mu)d\mu}{(2\hat{\eps}^{2}\cosh(\pi\mu)+1)^{k}}\sim
    \frac{(1-\beta)^{k-1}\hat{\eps}^{-2\Gb}}{\sin(\pi\beta)}\text{ as }\hat{\eps}\to 0^{+}. 
  \end{equation}
\end{lemma}
\begin{proof}
Changing the variable of integration $\mu'=\pi\mu+2\ln\hat{\Gve}$, we obtain
$$
\int_{0}^{\infty}\frac{e^{\pi\Gb\mu}v(\mu)d\mu}{(2\hat{\eps}^{2}\cosh(\pi\mu)+1)^{k}}= 
\frac{\hat{\eps}^{-2\Gb}}{\pi}\int_{2\ln\hat{\eps}}^{\infty}\frac{e^{\Gb\mu'}
v\left(\frac{\mu'}{\pi}-\frac{2}{\pi}\ln\hat{\eps}\right)}{(e^{\mu'}+e^{-\mu'+4\ln\hat{\eps}}+1)^{k}}d\mu'.
$$
Since $\Re\Gb\in(0,k)$ and $v(\cdot)$ is a bounded function, the Lebesgue
dominated convergence theorem applies, and we obtain\footnote{The
  formula is correct only for $k=1$ or 2. For general $k\in\bb{N}$, the correct \rhs\ is
    more complicated:
    $\frac{1}{(k-\Gb)B(k,\Gb-k)\sin(\pi(\Gb-k))}$, where
  $B(x,y)$ is the Euler beta function.}
$$
\lim_{\hat{\Gve}\to 0^+}\int_{2\ln\hat{\eps}}^{\infty}\frac{e^{\Gb\mu'}
v\left(\frac{\mu'}{\pi}-\frac{2}{\pi}\ln\hat{\eps}\right)}{(e^{\mu'}+e^{-\mu'+4\ln\hat{\eps}}+1)^{k}}d\mu'
=\int_{\bb{R}}\frac{e^{\Gb\mu'}d\mu'}{(e^{\mu'}+1)^{k}}= \frac{\pi (1 - \beta)^{k-1}}{\sin(\pi \beta)}.
$$
\end{proof}
As a corollary, we obtain the explicit asymptotics of 
$\psi_{\Gve}(z)$, $\|\psi_{\eps}\|_{2}$ and $\|\psi_{\eps}\|$.

\begin{theorem}
\label{th:psiyp}
Let $x_{0}>1$ and $\psi_{\Gve}$ be the solution of the integral equation
(\ref{inteq}). Formula (\ref{phirep}) defines an analytic extension of
$\psi_{\Gve}(x)$ from $[0,1]$ to the complex right half-plane.
Suppose $z\in\GO=\{z\in\bb{C}:\Re z>0,\ z\not\in[0,1]\}$. Then,
\begin{itemize}
\item[(i)] $\displaystyle\psi_{\Gve}(z)\sim\frac{R(x_0)R(z)}{2\pi\sin(\pi\Gb(z))}\hat{\Gve}^{-2\Gb(z)}$,
  where $\displaystyle\Gb(z)=\frac{\Ga(x_{0})+\Ga(z)}{\pi}$, and
  $\displaystyle\hat{\Gve}=\frac{\Gve}{\sqrt{2\pi}}$. 
\item[(ii)] $\displaystyle\|\psi_{\eps}\|_{2}\sim
  \nth{\pi}\sqrt{\frac{x_{0}\arcsin(1/x_{0})}{2(x_{0}^{2}-1)}}\hat{\Gve}^{-\Gb(x_{0})}$;
\item[(iii)] $\displaystyle\Gve\|\psi_{\eps}\|\sim\nth{\pi}\sqrt{\frac{x_0\arccos(1/x_{0})}{2(x_0^2 - 1)}}\hat{\Gve}^{-\Gb(x_{0})}$.
\end{itemize}
\end{theorem}
\begin{proof}
We begin by ``substituting'' our large $\mu$ asymptotics (\ref{uofx0})
from Lemma~\ref{lem:ux0asym} into formulas (\ref{phirep}),
(\ref{phinorm}), and (\ref{phiHnorm}). We obtain
\[
\psi_{\Gve}(z)=\frac{R(x_0)R(z)}{2\pi}\int_{0}^{\infty}\frac{e^{\pi\beta(z)\mu}v(z;\mu)}{2\hat{\Gve}^2\cosh(\pi \mu) + 1} d\mu,
\]
\[
\|\psi_{\Gve}\|_2^2=\frac{R(x_0)^2}{2\pi}\int_{0}^{\infty}\frac{e^{\pi\beta(x_0)\mu}v(x_0;\mu)}{(2\hat{\Gve}^2\cosh(\pi \mu) + 1)^2} d\mu,
\]
and
\[
\|\psi_{\Gve}\|^2=\frac{R(x_{0})^{2}}{4\pi^{2}}\int_{0}^{\infty}\frac{e^{(1+\Gb(x_{0}))\pi\mu}\Tld{v}(x_{0};\mu)}{(2\hat{\Gve}^2\cosh(\pi \mu) + 1)^2}d\mu,
\]
where
\[
v(z;\mu) = \frac{u(x_0;\mu)}{u_0(x_0;\mu)} \cdot
\frac{u(z;\mu)}{u_0(z;\mu)}\tanh(\pi \mu),\quad
\Tld{v}(z;\mu)=2v(z;\mu)e^{-\pi\mu}\cosh(\pi\mu),
\]
and
\begin{equation}
  \label{u0zmu}
  u_{0}(z;\mu)=R(z)\frac{e^{\mu\Ga(z)}}{\sqrt{2\pi\mu}}.
\end{equation}
In order to apply Lemma~\ref{lem:intasymp}, we must verify that the
function $\mu\mapsto v(z;\mu)$ and the exponent $\Gb(z)$ satisfy the
assumptions in  Lemma~\ref{lem:intasymp}. The continuity of $\mu\mapsto
v(z;\mu)$ follows from
the Euler integral representation of the hypergeometric function combined with formula
(\ref{uofx}), which gives
\begin{equation}
  \label{Euler}
  u(z;\mu)=z^{-\frac{1}{2}+i\mu}\frac{\sin\left(\frac{3\pi}{4}+i\frac{\pi\mu}{2}\right)}{\pi}
\int_0^1t^{-\frac{1}{4}+\frac{i\mu}{2}}(1-t)^{-\frac{3}{4}-\frac{i\mu}{2}}
(1-(1-z^2)t)^{-\frac{1}{4}-\frac{i\mu}{2}}dt.
\end{equation}
The integrand is continuous in $\mu$ and bounded by
$t^{-1/4}(1-t)^{-3/4}|(1-(1-z^2)t|^{-1/4}e^{\pi\mu} \in L^1(0,1)$. An
application of the Lebesgue dominated convergence theorem implies that
$\mu\mapsto u(z;\mu)$ is continuous on $[0,\infty)$ for any $z\in\GO$. Formula
(\ref{u0zmu}) shows that $u_0(z;\mu)$ is nonvanishing and continuous in
$\mu>0$ proving the continuity of $\mu\mapsto v(z;\mu)$, while
Lemma~\ref{lem:ux0asym} implies that $v(z;\mu)\to 1$, as $\mu\to\infty$, for
every $z\in\GO$. Finally, the required constraint $\Re\Gb(z)\in(0,1)$ for any
$z\in\GO$, is guaranteed by the following lemma.
\begin{lemma}\label{lem:expbnd}
 $\Re\Ga(z)\in(0,\frac{\pi}{2})$ for any $z\in\GO$, where $\Ga(z)$ is defined in (\ref{Ralpha}).
\end{lemma}
\begin{proof}
We observe that $\Ga:\GO\to\bb{C}$ is injective since
$\cos\Ga(z)=1/z$. Thus, $\Md_{\infty}\Ga(\GO)=\Ga(\Md_{\infty}\GO)$,
where $\Md_{\infty}\GO$ refers to the boundary of $\GO$ in the
Riemann sphere $\bb{C}\cup\{\infty\}$. It
is easy to see that $\Ga(z)$ maps the ray $i(0,+\infty)$ to the line
$\pi/2+i\bb{R}$; the ray $i(-\infty,0)$ to the same line
$\pi/2+i\bb{R}$. It maps the interval $[0,1]+0i$ to the ray
$i[0,+\infty]$ and the interval $[0,1]-0i$ to the ray
$i[-\infty,0]$. While, $\sqrt{z^{2}-1}=z\sqrt{1-z^{-2}}$, when $z\to\infty$,
$z\in\GO$. Therefore, $\Ga(z)\to i\ln(-i)=\pi/2$, as $z\to\infty$. We
conclude that
$\Md_{\infty}\Ga(\GO)=i\bb{R}\cup\pi/2+i\bb{R}\cup\{\infty\}$, and
$\Ga(\GO)=\{w\in\bb{C}:0<\Re w<\pi/2\}$ since
$\Ga(\GO)$ must be a connected subset of $\bb{C}$.
\end{proof}
Lemma~\ref{lem:intasymp} can now be applied, and we obtain
\[
\psi_{\Gve}(z)\sim\frac{R(x_0)R(z)\hat{\Gve}^{-2\Gb(z)}}{2\pi\sin\pi\Gb(z)},
\]
\[
\|\psi_{\Gve}\|_2^2\sim\frac{R(x_0)^2(1-\Gb(x_{0}))\hat{\Gve}^{-2\Gb(x_{0})}}
{2\pi\sin\pi\Gb(x_{0})},
\]
and
\[
\|\psi_{\Gve}\|^2\sim\frac{R(x_{0})^{2}\Gb(x_{0})\hat{\Gve}^{-2\Gb(x_{0})-2}}
{4\pi^{2}\sin\pi\Gb(x_{0})}.
\]
Substituting the values of $R(x_{0})$, $\Ga(x_{0})$, and $\Gb(x_{0})$, we
obtain the claimed asymptotic formulas (i)---(iii).
\end{proof}
Now, we can compute the explicit asymptotics of $\GD^{x_{0}}_{*}(\Ge)$, given
in (\ref{exactasymp1}). We compute
\[
\frac{\psi_{\Gve}(x_{0})}{\|\psi_{\Gve}\|}\sim\hat{\Gve}^{1-\Gb(x_{0})}\sqrt{\frac{x_{0}}{2(x_{0}^{2}-1)\Gb(x_{0})}}=:E_{0}(\Gve),
\]
and
\[
\frac{\|\psi_{\Gve}\|_{2}}{\|\psi_{\Gve}\|}\sim\Gve\sqrt{\frac{\arcsin(1/x_{0})}{\arccos(1/x_{0})}}=:E_{1}(\Gve).
\]
It is now evident that functions $E_{0}(\Gve)$ and $E_{1}(\Gve)$ are
continuous and monotone increasing on $[0,1)$, such that
$E_0(0)=E_1(0)=0$. Since $E_{1}(\Gve)$ is linear and $E_{0}(\Gve)$ is a
constant multiple of a power, property (\ref{E0E1prop}) reads
\[
\frac{E_0(\eps)}{E_0\left(E_1^{-1}\left(E_1(\eps) r(\eps) \right) \right)} = \frac{E_0(\eps)}{E_0\left(\eps r(\eps) \right)} = 
(r(\eps))^{\Gb(x_{0})-1} \to 1, \text{ as } \eps \to 0^+.
\]
for any function $r(\eps)$ such that $r(\eps) \to 1$ as $\eps \to 0^+$. Thus,
formula (\ref{exactasymp1}) applies, and
\[
  \GD^{x_{0}}_{*}(\Ge) \sim E_0(E_1^{-1}(\Ge))=\sqrt{\frac{x_{0}}{2(x_{0}^{2}-1)\Gb(x_{0})}}
\left(\sqrt{\frac{2\pi\arcsin(1/x_{0})}{\arccos(1/x_{0})}}\right)^{\Gb(x_{0})-1}\Ge^{1-\Gb(x_{0})}.
\]
Substituting the values of $\Ga(x_{0})=\arccos(1/x_{0})$ and
$\Gb(x_{0})=2\Ga(x_{0})/\pi$ into the above formula, we obtain
Theorem~\ref{th:GDstar} for all $x_{0}>1$. In particular, we see that for any $x_{0}>1$
\begin{equation}
  \label{gammastar}
  \Gg_{*}(x_{0})=\lim_{\Ge\to 0}\frac{\ln\GD^{x_{0}}_{*}(\Ge)}{\ln\Ge}=\frac{2}{\pi}\arcsin\left(\frac{1}{x_{0}}\right).
\end{equation}

The singular behavior at $x_{0}=1$ of coefficients in all of our asymptotic
formulas indicates that the asymptotic analysis for $x_{0}=1$ needs to be done
separately.

\begin{theorem}
\label{th:psi1}
Let $\psi_{\Gve}$ be the solution of the integral equation (\ref{inteq}) with
$x_{0}=1$. Then
\begin{itemize}
\item[(i)] $\displaystyle\| \psi_{\eps} \|_2^2 \sim \psi_{\eps}(1) \sim
  \frac{2(\ln \eps)^2}{\pi^2}$;
\item[(ii)] $\displaystyle\| \psi_{\eps} \|^2 \sim \frac{-2\ln\Gve}{\pi^{2} \eps^2}$.
\end{itemize}
\end{theorem}
\begin{proof}

Whenever $x_0 = 1$, our formulas (\ref{phinorm}), (\ref{phix0}), and
(\ref{phiHnorm}) simplify because $u(1;\mu) \equiv 1$:
\begin{equation}\label{psi1exprss}
\psi_{\Gve}(1) =  \int_{0}^{\infty}\frac{\mu \tanh(\pi \mu) }{2 \hat{\eps}^2 \cosh(\pi \mu) +1}d\mu, 
\quad
\| \psi_{\Gve} \|_2^2  =  \int_{0}^{\infty}\frac{\mu \tanh(\pi \mu) }{(2 \hat{\eps}^2 \cosh(\pi \mu) +1)^2}d\mu,
\end{equation}
\begin{equation}
  \label{psix1Hnorm}
  \|\psi_{\Gve}\|^{2}=\nth{\pi}\int_{0}^{\infty} \frac{\mu\sinh(\pi \mu)}{(2\hat{\Gve}^{2}\cosh(\pi\mu)+1)^{2}} d\mu. 
\end{equation}
The situation here is similar to the one for $x_{0}>1$ in that setting
$\hat{\Gve}=0$ still results in divergent integrals. This indicates that it is the
behavior of the integrands at $\mu=\infty$ that determines the asymptotics of
the integrals when $\hat{\Gve}\to 0^{+}$. When $\mu$ is large $\tanh(\pi\mu)$
will be replaced by 1, and both $2\cosh(\pi\mu)$ and $2\sinh(\pi\mu)$, by
$e^{\pi\mu}$. To make this heuristic argument rigorous, we make a simple observation
that we formulate as a lemma for easy reference.
\begin{lemma}\label{lem:W}
Let $(G,\Gs)$ be an arbitrary measure space. Suppose that for any $\eps\in(0,\eps_{0})$
$\{W_{\eps},\hat{W}_{\eps}\}\subset L^1(G;d\Gs)$ and
\begin{itemize}
\item[(i)] $\displaystyle\limi_{\eps\to 0^{+}}\left|\int_{G}W_{\eps}(\mu)d\Gs(\mu)\right|=\infty$;
\item[(ii)] $\displaystyle\lims_{\eps\to 0^{+}}\|W_{\eps}-\hat{W}_{\eps}\|_{L^{1}(G;d\Gs)}<\infty$.
\end{itemize}
Then $\int_G W_{\eps}(\mu)d\Gs(\mu) \sim \int_G \hat{W}_{\eps}(\mu)d\Gs(\mu)$,
as $\eps\to 0^{+}$.
\end{lemma}
\begin{proof}
$\displaystyle
\lims_{\eps\to 0^{+}}\left| \frac{\int_G \hat{W}_{\eps}(\mu)d\Gs(\mu)}{\int_G W_{\eps}(\mu) d\Gs(\mu)} - 1 \right| \le \frac{\displaystyle\lims_{\eps\to 0^{+}}\|\hat{W}_{\eps} - W_{\eps} \|_{L^1(G;d\Gs)}}{\displaystyle\limi_{\eps\to 0^{+}}\left|\int_G W_{\eps}(\mu)d\Gs(\mu)\right|}= 0.
$
\end{proof}
As we have already pointed out, the integrals in (\ref{psi1exprss}) and
(\ref{psix1Hnorm}) satisfy condition (i) of the lemma. Then estimates 
\[
|\tanh(\pi\mu)-1|\le2e^{-2\pi\mu},\qquad|2\sinh(\pi\mu)-e^{\pi\mu}|=e^{-\pi\mu}
\]
ensure that condition (ii) of the lemma is satisfied, and we conclude that
\[
\psi_{\Gve}(1)\sim\int_{0}^{\infty}\frac{\mu d\mu}{2 \hat{\eps}^2 \cosh(\pi \mu) +1}, 
\quad
\| \psi_{\Gve} \|_2^2\sim\int_{0}^{\infty}\frac{\mu d\mu}{(2 \hat{\eps}^2 \cosh(\pi \mu) +1)^2},
\]
and
\[
  \|\psi_{\Gve}\|^{2}\sim\nth{2\pi}\int_{0}^{\infty} \frac{\mu e^{\pi\mu}d\mu}{(2\hat{\Gve}^{2}\cosh(\pi\mu)+1)^{2}}. 
\]
Similarly, the estimate
\[
\left|\frac{\mu}{2\hat{\Gve}^{2}\cosh(\pi\mu)+1}-\frac{\mu}{\hat{\Gve}^{2}e^{\pi\mu}+1}\right|  \le  2\hat{\eps}^{2}\mu e^{-\pi\mu}
\]
implies that
\[
\psi_{\Gve}(1)\sim\int_{0}^{\infty}\frac{\mu d\mu}{\hat{\eps}^2e^{\pi\mu}+1}.
\]
To handle the remaining two integrals, we define
\[
W_{\Gve}(\mu)=\frac{\mu e^{\pi\mu}}{(2\hat{\Gve}^{2}\cosh(\pi\mu)+1)^2},\quad
\hat{W}_{\Gve}(\mu)=\frac{\mu e^{\pi\mu}}{(\hat{\Gve}^{2}e^{\pi\mu}+1)^2}.
\]
We first compute
\[
|W_{\Gve}(\mu)-\hat{W}_{\Gve}(\mu)|=\frac{\mu\hat{\Gve}^{2}(2\hat{\Gve}^{2}e^{\pi\mu}+2+\hat{\Gve}^{2}e^{-\pi\mu})}{(2\hat{\Gve}^{2}\cosh(\pi\mu)+1)^2(\hat{\Gve}^{2}e^{\pi\mu}+1)^2},
\]
and estimate
\[
2\hat{\Gve}^{2}e^{\pi\mu}+2+\hat{\Gve}^{2}e^{-\pi\mu}\le3(\hat{\Gve}^{2}e^{\pi\mu}+1),
\]
so that
\[
|W_{\Gve}(\mu)-\hat{W}_{\Gve}(\mu)|
\le \frac{3\mu\hat{\Gve}^{2}}{(2\hat{\Gve}^{2}\cosh(\pi\mu)+1)^2(\hat{\Gve}^{2}e^{\pi\mu}+1)}.
\]
Next, we estimate
\[
\hat{\Gve}^{2}e^{\pi\mu}+1\ge \hat{\Gve}^{2}e^{\pi\mu},\qquad
(2\hat{\Gve}^{2}\cosh(\pi\mu)+1)^2\ge 1,
\]
and obtain
\[
|W_{\Gve}(\mu)-\hat{W}_{\Gve}(\mu)|\le 3\mu e^{-\pi\mu}\in L^{1}(0,\infty).
\]
Thus, Lemma~\ref{lem:W} is applicable and
\[
\psi_{\Gve}(1)\sim\int_{0}^{\infty}\frac{\mu d\mu}{\hat{\eps}^2e^{\pi\mu}+1}:=I_{1}(\hat{\Gve}), 
\qquad
\| \psi_{\Gve} \|_2^2\sim\int_{0}^{\infty}\frac{\mu
  d\mu}{(\hat{\eps}^2e^{\pi\mu}+1)^2}:=I_{2}(\hat{\Gve}),
\]
\[
\|\psi_{\Gve}\|^{2}\sim I_{0}(\hat{\Gve}):=\nth{2\pi}\int_{0}^{\infty}\frac{\mu e^{\pi\mu}d\mu}
{(\hat{\Gve}^{2}e^{\pi\mu}+1)^{2}}=
\frac{\ln\left(1+\frac{1}{\hat{\Gve}^{2}}\right)}{2\pi^{3}\hat{\Gve}^{2}}\sim
-\frac{\ln\hat{\Gve}}{\pi^{3}\hat{\Gve}^{2}},
\]
establishing part (ii) of the theorem. Part (i) is proved by means of the
L'H\^opital rule:
\[
\lim_{\hat{\Gve}\to 0^{+}}\frac{I_{1}(\hat{\Gve})}{(\ln\hat{\Gve})^{2}}=
\lim_{\hat{\Gve}\to 0^{+}}\frac{\hat{\Gve}I'_{1}(\hat{\Gve})}{2\ln\hat{\Gve}}=
-\lim_{\hat{\Gve}\to 0^{+}}\frac{2\pi\hat{\Gve}^{2}I_{0}(\hat{\Gve})}{\ln\hat{\Gve}}=\frac{2}{\pi^{2}}.
\]
To apply the L'H\^opital rule to $I_{2}(\hat{\Gve})$, we compute
\[
I'_{2}(\hat{\Gve})=-4\hat{\Gve}\int_{0}^{\infty}\frac{\mu e^{\pi\mu}d\mu}
{(\hat{\Gve}^{2}e^{\pi\mu}+1)^{3}}=
-2\frac{(\hat{\Gve}^{2}+1)\ln(1+\hat{\Gve}^{-2})-1}{\pi^{2}\hat{\Gve}(\hat{\Gve}^{2}+1)}
\sim\frac{4\ln\hat{\Gve}}{\pi^{2}\hat{\Gve}}.
\]
Thus,
\[
\lim_{\hat{\Gve}\to 0^{+}}\frac{I_{2}(\hat{\Gve})}{(\ln\hat{\Gve})^{2}}=
\lim_{\hat{\Gve}\to 0^{+}}\frac{\hat{\Gve}I'_{2}(\hat{\Gve})}{2\ln\hat{\Gve}}=\frac{2}{\pi^{2}}.
\]
The theorem is now proved.
\end{proof}
According to Theorem~\ref{th:psi1},
\[
\frac{\psi_{\eps}(1)}{\|\psi_{\eps}\|}\sim\frac{\sqrt{2}}{\pi}\Gve|\ln\Gve|^{3/2}=:E_{0}(\Gve),\quad
\frac{\|\psi_{\eps}\|_{2}}{\|\psi_{\eps}\|}\sim\Gve\sqrt{|\ln\Gve|}=:E_{1}(\Gve).
\]
This shows that both $E_{0}(\Gve)$ and $E_{1}(\Gve)$ are continuous, monotone
increasing functions on $[0,e^{-3/2})$, satisfying $E_{0}(0)=E_{1}(0)=0$. In order to use formula
(\ref{exactasymp1}) for the exact asymptotics of $\GD^{1}_{*}(\Ge)$, we need to
verify property (\ref{E0E1prop}). This is somewhat tedious. Let
$r(\Gve)\to 1$, as $\Gve\to 0^{+}$ be arbitrary. To make calculations more
compact, we define $\Gd=\Gd(\Gve)=r(\Gve)E_{1}(\Gve)$. Then,
\[
\rho(\Gve)=\frac{E_{0}(\Gve)}{E_{0}(E_{1}^{-1}(E_{1}(\Gve)r(\Gve)))}=
\frac{\Gve|\ln\Gve|^{3/2}}{E_{1}^{-1}(\Gd)|\ln E_{1}^{-1}(\Gd)|^{3/2}}=
\frac{\Gve|\ln\Gve|^{3/2}}{\Gd|\ln E_{1}^{-1}(\Gd)|}=\frac{|\ln\Gve|}{r(\Gve)|\ln E_{1}^{-1}(\Gd)|},
\]
where we have used the relation $E_{1}^{-1}(\Gd)|\ln
E_{1}^{-1}(\Gd)|^{1/2}=E_{1}(E_{1}^{-1}(\Gd))=\Gd$ together with the formula
for $\Gd(\Gve)$. Next, we write
\[
|\ln\Gve|=|\ln r(\Gve)E_{1}(\Gve)-\ln(r(\Gve)\sqrt{|\ln\Gve|})|=|\ln \Gd(\Gve)|\tilde{r}(\Gve),
\]
where
\[
\tilde{r}(\Gve)=\left|1-\frac{\ln(r(\Gve)\sqrt{|\ln\Gve|})}{\ln(r(\Gve)\Gve\sqrt{|\ln\Gve|})}\right|\to
1,\text{ as }\Gve\to 0^{+}.
\]
Thus,
\[
\rho(\Gve)=\frac{\tilde{r}(\Gve)|\ln\Gd|}{r(\Gve)|\ln E_{1}^{-1}(\Gd)|}.
\]
It remained to observe that $\Gd(\Gve)\to 0^{+}$, as $\Gve\to 0^{+}$ and
therefore, $\eta(\Gve)=E_{1}^{-1}(\Gd(\Gve))\to 0^{+}$. Hence,
\[
\lim_{\Gve\to 0^{+}}\rho(\Gve)=\lim_{\Gve\to 0^{+}}\frac{\tilde{r}(\Gve)}{r(\Gve)}\cdot
\lim_{\Gd\to 0^{+}}\frac{|\ln\Gd|}{|\ln E_{1}^{-1}(\Gd)|}=\lim_{\eta\to 0^{+}}
\frac{|\ln E_{1}(\eta)|}{|\ln\eta|}=1.
\]
Formula (\ref{exactasymp1}) is now applicable, and we compute, using $E_{1}^{-1}(\Ge)|\ln
E_{1}^{-1}(\Ge)|^{1/2}=\Ge$,
\[
\GD^{1}_{*}(\Ge)\sim\frac{\sqrt{2}}{\pi}E_{1}^{-1}(\Ge)|\ln E_{1}^{-1}(\Ge)|^{3/2}=
\frac{\sqrt{2}}{\pi}\Ge|\ln\Ge|\frac{|\ln E_{1}^{-1}(\Ge)|}{|\ln\Ge|}\sim
\frac{\sqrt{2}}{\pi}\Ge|\ln\Ge|
\]
since
\[
\lim_{\Ge\to 0^{+}}\frac{|\ln E_{1}^{-1}(\Ge)|}{|\ln\Ge|}=\lim_{\eta\to 0^{+}}\frac{|\ln\eta |}{|\ln E_{1}(\eta)|}=1.
\]
In particular, we can conclude that
\[
\Gg_{*}(1)=\lim_{\Ge\to 0^{+}}\frac{\ln\GD^{1}_{*}(\Ge)}{\ln\Gve}=1=\lim_{x_{0}\to 1^{+}}\Gg_{*}(x_{0}).
\]
This completes the proof of Theorem~\ref{th:GDstar} for $x_{0}=1$.


\section{A continuous family of Hilbert space norms}
\setcounter{equation}{0}
Our task now is to connect the explicit exponent $\Gg_{*}(x_{0})$, given by (\ref{gammastar})
to the desired exponent $\Gg(x_{0})$ coming from the $(f,g)$-problem
(\ref{fgproblem}). This is done by introducing a family of norms that help us
to bridge the gap between the the $L^{2}(0,1)$ norm and the $H^{2}(\CR)$ norm
on the convex cone $\mathfrak{C}_{2}$.
In reference to $f\in\mathfrak{C}_{2}$, we will use the notation
\begin{equation}
  \label{Hrep}
  (\CF_{p}[f])(z)=\frac{f(z^{1/p})}{z^{\frac{p-1}{2p}}(z^{1/p}+1)},\quad p\ge 1,
\end{equation}
where the principal branch of $z^{\Ga}$ is always chosen. For all $p\ge 1$
and $f\in\mathfrak{C}_{2}$ the functions $\CF_{p}[f]$ are analytic on the
complex right half-plane $\CR$. We
then define the family of spaces
\begin{equation}
  \label{Hardy}
  \mathfrak{H}_{p}=\{f\in\CH(\CR): \CF_{p}[f]\in H\},\quad p\ge 1,
\end{equation}
equipped with norms
$
\|f\|_{\mathfrak{H}_{p}}=\|\CF_{p}[f]\|.
$
\begin{theorem}
  \label{th:HLB}
$\mathfrak{C}_{2}\subset \mathfrak{H}_{p}$, for every $p>1$, and there exists a constant $C_{p}>0$, such that
  \begin{equation}
    \label{HLB}
    \|f\|_{\mathfrak{H}_{p}}\le C_{p}\|f\|_{2},
  \end{equation}
for every $f\in\mathfrak{C}_{2}$.
\end{theorem}
\begin{proof}
  The fact that the functions $f_{\Gd_{t}}=e^{-xt}$ belong to $\mathfrak{H}_{p}$
  follows from the observations that for each fixed $y>0$ and $q\in[0,1]$ the functions
  \[
    \nu_{1}(x)=\re\left[(x+iy)^{q}\right],\quad\nu_{2}(x)=|(x+iy)^{q}|,\quad
    \nu_{3}(x)=|(x+iy)^{q}+1|^{2},\quad q\in[0,1]
  \]
are monotone increasing in $x\in(0,+\infty)$. This is evident from the
polar representation of $x+iy=r(x)e^{i\Gth(x)}$ and the observation
that $r(x)$ is an increasing function of $x$, while $\Gth(x)\in(0,\pi/2)$ is a
decreasing one. Then
\[
  \nu_{1}(x)=r(x)^{q}\cos(q\Gth(x)),\quad\nu_{2}(x)=r(x)^{q},\quad
  \nu_{3}(x)=r(x)^{2q}+1+2r(x)^{q}\cos(q\Gth(x))
\]
are obviously increasing functions since $q\Gth(x)\in[0,\pi/2]$ for
all $x\ge 0$ and $q\in[0,1]$.
Thus,
\[
  |(\CF_{p}[f_{\Gd_{t}}])(x+iy)|^{2}=\frac{e^{-2t\nu_{1}(x)}}{\nu_{2}(x)\nu_{3}(x)}\le
\frac{e^{-2t\nu_{1}(0)}}{\nu_{2}(0)\nu_{3}(0)}=|(\CF_{p}[f_{\Gd_{t}}])(iy)|^{2}.
\]
It is also easy to see that
\[
\int_{0}^{\infty}|(\CF_{p}[f_{\Gd_{t}}])(iy)|^{2}dy=\int_{0}^{\infty}\frac{e^{-2ta_{p}y^{1/p}}}{y^{\frac{p-1}{p}}(y^{2/p}+1+y^{1/p}a_{p})}=p\int_{0}^{\infty}\frac{e^{-2ta_{p}u}}{u^{2}+1+2a_{p}u}du<\infty,
\]
where $a_{p}=\cos(\pi/(2p))$. We conclude that 
\begin{equation}
  \label{flcoe}
    f(x)=\sum_{j=1}^{N}c_{j}e^{-xt_{j}}\in\mathfrak{H}_{p}
\end{equation}
for all $p\ge 1$.

Now, let $\Gs$ be a positive measure, such that
$f_{\Gs}\in\mathfrak{H}_{p}\cap\mathfrak{C}_{2}$. Let us show that (\ref{HLB})
holds for all such functions $f_{\Gs}$. Indeed, for any
$f_{\Gs}\in\mathfrak{H}_{p}\cap\mathfrak{C}_{2}$, we have (see (\ref{norm}))
\[
\|f_{\Gs}\|_{\mathfrak{H}_{p}}=\nth{\sqrt{\pi}}\|(\CF_{p}[f_{\Gs}])(iy)\|_{L^{2}(0,\infty)}.
\]
Then, in order to establish (\ref{HLB}), we need to prove the inequality
\begin{equation}
  \label{L2Hardy}
  \|(\CF_{p}[f_{\Gs}])(iy)\|_{L^{2}(0,\infty)}\le\sqrt{\pi}C_{p}\|f_{\Gs}\|_{2}.
\end{equation}

To prove (\ref{L2Hardy}), we estimate
\begin{equation}
  \label{fzest}
  \left|f_{\Gs}\left((iy)^{1/p}\right)\right|\le\int_{0}^{\infty}e^{-a_{p}y^{1/p}t}d\Gs(t)
  =f_{\Gs}\left(a_{p}y^{1/p}\right).
\end{equation}
We conclude that
\[
\|(\CF_{p}[f_{\Gs}])(iy)\|_{L^{2}(0,\infty)}^{2}\le\int_{0}^{\infty}
\frac{\left|f_{\Gs}\left(a_{p}y^{1/p}\right)\right|^{2}}
{y^{\frac{p-1}{p}}|i^{1/p}y^{1/p}+1|^{2}}dy.
\]
Making a change of variables $u=a_{p}y^{1/p}$, we
obtain
\[
\|(\CF_{p}[f_{\Gs}])(iy)\|_{L^{2}(0,\infty)}^{2}\le\frac{p}{a_{p}}
\int_{0}^{\infty}\frac{|f_{\Gs}(u)|^{2}}{(u+1)^{2}+u^{2}\tan^{2}\left(\frac{\pi}{2p}\right)}du.
\]
Writing
\[
\int_{0}^{\infty}\frac{|f_{\Gs}(u)|^{2}}{(u+1)^{2}+u^{2}\tan^{2}\left(\frac{\pi}{2p}\right)}du=
\sum_{n=0}^{\infty}\int_{n}^{n+1}\frac{|f_{\Gs}(u)|^{2}}{(u+1)^{2}+u^{2}\tan^{2}\left(\frac{\pi}{2p}\right)}du,
\]
and estimating
\[
\nth{(u+1)^{2}+u^{2}\tan^{2}\left(\frac{\pi}{2p}\right)}\le
\nth{(n+1)^{2}+n^{2}\tan^{2}\left(\frac{\pi}{2p}\right)},
\]
when $u\in[n,n+1]$, we obtain the bound
\[
\|(\CF_{p}[f_{\Gs}])(iy)\|_{L^{2}(0,\infty)}^{2}\le\frac{p}{a_{p}}
\sum_{n=0}^{\infty}\frac{\int_{0}^{1}|f_{\Gs}(x+n)|^{2}dx}{(n+1)^{2}+n^{2}\tan^{2}\left(\frac{\pi}{2p}\right)}.
\]
Finally, using the fact that $0\le f_{\Gs}(x+n)\le f_{\Gs}(x)$ for any CMF
 $f_{\Gs}$, we conclude that
\[
\|(\CF_{p}[f_{\Gs}])(iy)\|_{L^{2}(0,\infty)}^{2}\le\frac{p\|f_{\Gs}\|_{2}^{2}}{a_{p}}
\sum_{n=0}^{\infty}\frac{1}{(n+1)^{2}+n^{2}\tan^{2}\left(\frac{\pi}{2p}\right)}.
\]
If we replace $n+1$ by $n$ in the bound above for $n>0$, we obtain a simpler
formula for the constant $C_{p}$:
\[
C_{p}^{2}=\frac{p}{\pi a_{p}}+\frac{\pi pa_{p}}{6},\qquad 
a_{p}=\cos\left(\frac{\pi}{2p}\right).
\]
To finish the proof of the theorem, we need the following density lemma.
\begin{lemma}
  \label{lem:C2density}
  Suppose $f\in\mathfrak{C}_{2}$. Then there exists a sequence of functions
  $f_{n}\in\mathfrak{C}_{2}$ of the form (\ref{flcoe}), such that $f_{n}\to f$ in $L^{2}(0,1)$.
\end{lemma}
\begin{proof}
Let $K$ be the closure in $L^{2}(0,1)$ of the set of positive finite linear
combinations of functions $f_{\Gd_{t}}(x)=e^{-xt}$. Then, $K$ is a closed,
convex subset of $L^{2}(0,1)$. Suppose, there exists
$f_{0}\in\mathfrak{C}_{2}\setminus K$. Then, by the Hahn-Banach
separation theorem there exists $g_{0}\in
L^{2}(0,1)$, such that for all $t\ge 0$
\[
\int_{0}^{1}e^{-xt}g_{0}(x)dx\ge 0>\int_{0}^{1}f_{0}(x)g_{0}(x)dx.
\]
If $\Gs_{0}$ is the spectral measure of $f_{0}\in\mathfrak{C}_{2}$, then integrating the
left inequality above with respect to $\Gs_{0}$, we obtain
\[
\int_{0}^{1}f_{0}(x)g_{0}(x)dx\ge 0,
\]
which contradicts the right inequality. We conclude that $K=\mathfrak{C}_{2}$.
\end{proof}
Now, if $f\in\mathfrak{C}_{2}$ and $f_{n}=f_{\Gs_{n}}$ is as in the
lemma, then by Lemma~\ref{lem:L2star}
$
  \|\Gs_{n}\|_{*}\le \|f_{n}\|_{2}.
$
Thus, we can extract a weak-* convergent subsequence in $X^{*}$, not relabeled, so that
$\Gs_{n}\wk{*}\Gs$, where $X$ is defined in (\ref{X}). It follows that along this subsequence $f_{\Gs_{n}}(z)\to
f_{\Gs}(z)$ for all $z\in\CR$ since $e^{-zt}\in X$. Thus, since $f_{\Gs_{n}}\to f$ in $L^{2}(0,1)$,
then $f_{\Gs}=f$, and consequently $(\CF_{p}[f_{n}])(z)\to
(\CF_{p}[f])(z)$ pointwise on $\CR$. In addition, by the already proved
inequality (\ref{HLB}) for functions (\ref{flcoe}), we have
$\|\CF_{p}[f_{n}]\|=\|f_{n}\|_{\mathfrak{H}_{p}}\le
C_{p}\|f_{n}\|_{2}$. Hence, there exists a further subsequence, not relabeled,
along which $\CF_{p}[f_{n}]\weak F$ in $H^{2}(\CR)$. But, since $H^{2}(\CR)$
is a reproducing kernel Hilbert space, weak convergence implies pointwise
convergence, showing that $\CF_{p}[f]=F\in H$. We conclude that
$f\in\mathfrak{H}_{p}$, and the theorem is now proved.
\end{proof}
We emphasize that inequality (\ref{HLB}) is valid only for all
$f\in\mathfrak{C}_{2}$. It does not hold for
$f\in\mathfrak{X}=\mathfrak{C}_{2}-\mathfrak{C}_{2}$. In fact, our next theorem
establishes the reverse inequality.
\begin{theorem}
For every $p\ge 1$
\begin{equation}
  \label{HUB}
  \|f\|_{2}\le2\sqrt{\frac{2\pi}{p}}\|f\|_{\mathfrak{H}_{p}}
\end{equation}
for every $f\in\mathfrak{X}$ (every $f\in\mathfrak{X}\cap\mathfrak{H}_{1}$, if $p=1$).
\end{theorem}
\begin{proof}
  To prove this theorem, we use the analyticity of $(\CF_{p}[f])(z)$ in the right
  half-plane. Let $\GG_{L}$ be the boundary of the rectangle
  $[0,1]\times[0,L]$ traversed in the positive direction. We first observe
  that similarly to (\ref{fzest}), we can estimate 
\[
|f_{\Gs}((x+iL)^{1/p})|\le
f_{|\Gs|}\left(L^{1/p}a_{p}\right)\le f_{|\Gs|}(a_{p}).
\]
We conclude that 
\[
\lim_{L\to\infty}\int_{0}^{1}|(\CF_{p}[f_{\Gs}])(x+iL)|^{2}dx=0,
\]
and using the Cauchy theorem $\int_{\GG_{L}}(\CF_{p}[f_{\Gs}])(z)^{2}dz=0$, we obtain the formula
\[
\|\CF_{p}[f_{\Gs}]\|_{2}^{2}=\int_{0}^{1}(\CF_{p}[f_{\Gs}])(x)^{2}dx=\int_{0}^{\infty}(\CF_{p}[f_{\Gs}])(iy)^{2}idy
-\int_{0}^{\infty}(\CF_{p}[f_{\Gs}])(1+iy)^{2}idy.
\]
By the symmetry of CMFs, we have
$\bra{(\CF_{p}[f_{\Gs}])(z)}=(\CF_{p}[f_{\Gs}])(\bra{z})$. Therefore, we obtain the inequality
\begin{multline*}
  \|\CF_{p}[f_{\Gs}]\|_{2}^{2}\le\hf\int_{\bb{R}}|(\CF_{p}[f_{\Gs}])(iy)|^{2}dy+\hf\int_{\bb{R}}|(\CF_{p}[f_{\Gs}])(1+iy)|^{2}dy\\
  \le\int_{\bb{R}}|(\CF_{p}[f_{\Gs}])(iy)|^{2}dy=2\pi\|f_{\Gs}\|_{\mathfrak{H}_{p}}^{2}, 
\end{multline*}
where we used the property of Hardy functions that $\int_{\bb{R}}|F(x+iy)|^{2}dy$ is a
non-increasing function of $x$. Finally, changing
variable $u=x^{1/p}$, we estimate
\[
\|\CF_{p}[f_{\Gs}]\|_{2}^{2}=\int_{0}^{1}(\CF_{p}[f_{\Gs}])(x)^{2}dx=p\int_{0}^{1}\frac{f_{\Gs}(u)^{2}}{(u+1)^{2}}du\ge
\frac{p}{4}\|f_{\Gs}\|_{2}^{2}.
\]
\end{proof}
Now, in reference to the $\|\cdot\|_{\mathfrak{H}_{p}}$ norm, we can define the
$\phi_{p}$-problem by analogy with the $\phi$-problem (\ref{phiproblem}):
\begin{equation}
  \label{phip}
  \GD^{x_{0}}_{p}(\Ge)=\sup_{\phi\in\CA^{p}_{\Ge}}\phi(x_{0}),\quad
\CA_{\Ge}^{p}=\{\phi\in\mathfrak{H}_{p}:\|\phi\|_{\mathfrak{H}_{p}}\le 1,\|\phi\|_{2}\le\Ge\}.
\end{equation}


\section{The relations between $(f,g)$, $\phi$ and $\phi_{p}$-problems}
\setcounter{equation}{0} 
In this section, we are going to examine the relations between the
$(f,g)$, $\phi$ and $\phi_{p}$ problems, given by (\ref{fgproblem}), (\ref{phiproblem}),
and (\ref{phip}), respectively, with the goal of establishing (\ref{goal}),
thereby proving Theorem~\ref{th:main}.

Let $p>1$, and let $\psi^{(n)}_{\Ge}\in \mathfrak{H}_{p}$ be a
maximizing sequence in the $\phi_{p}$-problem (\ref{phip}). Define
$\phi^{(n)}_{\Ge}=\CF_{p}[\psi^{(n)}_{\Ge}]\in H$. Then
$\|\phi_{\Ge}^{(n)}\|=\|\psi_{\Ge}^{(n)}\|_{\mathfrak{H}_{p}}\le 1$, while 
\[
  \|\phi^{(n)}_{\Ge}\|_{2}^{2}=\int_{0}^{1}\frac{|\psi^{(n)}_{\Ge}(x^{1/p})|^{2}}{x^{(p-1)/p}(1+x^{1/p})^{2}}dx=p\int_{0}^{1}\frac{|\psi^{(n)}_{\Ge}(u)|^{2}}{(1+u)^{2}}du\le
  p\|\psi^{(n)}_{\Ge}\|_{2}^{2}\le p\Ge^{2}.
\]
Thus, $\phi^{(n)}_{\Ge}/\sqrt{p}$ is a valid test function for the $\phi$-problem, for every
$n\ge 1$, where $x_{0}$ was replaced by $x_{0}^{p}$. Therefore,
\begin{equation}
  \label{pstUB}
  \GD_{*}^{x_{0}^{p}}(\Ge)\ge\frac{\phi^{(n)}_{\Ge}(x_{0}^{p})}{\sqrt{p}}=
\frac{\psi^{(n)}_{\Ge}(x_{0})}{\sqrt{p}x_{0}^{(p-1)/2}(1+x_{0})}\to
\frac{\GD^{x_{0}}_{p}(\Ge)}{\sqrt{p}x_{0}^{(p-1)/2}(1+x_{0})},\text{ as }n\to\infty.
\end{equation}
Now, let $(f_{\Ge},g_{\Ge})$ be the solution of the $(f,g)$-problem. Define
$\phi_{\Ge}(x)=\GD[f_{\Ge},g_{\Ge}](x)$ (see (\ref{GDfg}) for notation). Then, $\|\phi_{\Ge}\|_{2}\le\Ge$, and by Theorem~\ref{th:HLB}
\[
\|\phi_{\Ge}\|_{\mathfrak{H}_{p}}\le\frac{\|f_{\Ge}\|_{\mathfrak{H}_{p}}+\|g_{\Ge}\|_{\mathfrak{H}_{p}}}
{\|f_{\Ge}\|_{2}+\|g_{\Ge}\|_{2}}\le C_{p}.
\]
Thus, $\phi_{\Ge}/(C_{p}+1)$ is a valid test function in the
$\phi_{p}$-problem for any $p>1$. Therefore,
\begin{equation}
  \label{gammaLB}
\GD^{x_{0}}_{p}(\Ge)\ge\frac{\phi_{\Ge}(x_{0})}{C_{p}+1}=\frac{\GD[f_{\Ge},g_{\Ge}](x_{0})}{C_{p}+1}=
\frac{\GD^{x_{0}}(\Ge)}{C_{p}+1}.
\end{equation}

An essential benefit of using the Hardy norm $\|\cdot\|$ is that it permits a controlled
split of functions $\phi\in H$ into the difference of two CMFs. Here
is the construction. By Lemma~\ref{lem:Hrep}, if $\phi\in H$, then there is a
unique $\Gs\in L^{2}(0,\infty)$, such that
\[
\phi(z)=\int_{0}^{\infty}e^{-zt}\Gs(t)dt,\quad\Re z>0.
\]
Let $\Gs_{+}(t)=\max\{0,\Gs(t)\}$, $\Gs_{-}(t)=\max\{0,-\Gs(t)\}$.
Then, we define
\[
\phi_{\pm}(z)=\int_{0}^{\infty}e^{-zt}\Gs_{\pm}(t)dt,\quad\Re z>0.
\]
In this construction 
\[
\int_{0}^{\infty}\Gs_{+}(t)\Gs_{-}(t)dt=0.
\]
Therefore, by Plancherel's identity
\[
\int_{\bb{R}}\phi_{+}(iy)\bra{\phi_{-}(iy)}dy=0.
\]
But then
\[
\|\phi\|^{2}=\nth{2\pi}\int_{\bb{R}}|\phi_{+}(iy)-\phi_{-}(iy)|^{2}dy=\|\phi_{+}\|^{2}+
\|\phi_{-}\|^{2}\ge\nth{4}(\|\phi_{+}\|+
\|\phi_{-}\|)^{2},
\]
which shows that
\[
\|\phi\|\le\|\phi_{+}\|+\|\phi_{-}\|\le2\|\phi\|.
\]

In order to complete the circle of inequalities, we take $\phi_{\Ge}$ to be the
solution of the $\phi$-problem
and define $f=\phi_{\Ge}^{+}$, $g=\phi_{\Ge}^{-}$. We then have, using Lemma~\ref{lem:HL2},
\[
  |\GD[f,g](x_{0})|\ge\frac{|\phi_{\Ge}(x_{0})|}{\|\phi_{\Ge}^{+}\|_{2}+\|\phi_{\Ge}^{-}\|_{2}}\ge
  \frac{|\phi_{\Ge}(x_{0})|}{\sqrt{\pi}(\|\phi_{\Ge}^{+}\|+\|\phi_{\Ge}^{-}\|)}\ge
  \frac{|\phi_{\Ge}(x_{0})|}{2\sqrt{\pi}\|\phi_{\Ge}\|}\ge\frac{|\phi_{\Ge}(x_{0})|}{2\sqrt{\pi}}=\frac{\GD^{x_{0}}_{*}(\Ge)}{2\sqrt{\pi}}.
\]
We also estimate
\[
  \|\GD[f,g](x)\|_{2}=\frac{\|\phi_{\Ge}\|_{2}}{\|\phi_{\Ge}^{+}\|_{2}+\|\phi_{\Ge}^{-}\|_{2}}
  \le\frac{C_{p}\Ge}{\|\phi_{\Ge}^{+}\|_{\mathfrak{H}_{p}}+\|\phi_{\Ge}^{-}\|_{\mathfrak{H}_{p}}}\le
  \frac{C_{p}\Ge}{\|\phi_{\Ge}\|_{\mathfrak{H}_{p}}}.
\]
To complete the circle of inequalities, we need the following theorem.
\begin{theorem}
\label{th:phiHp}
For any $x_0 \ge 1$ and $p>1$, there exists $c_{p}(x_{0})>0$, such that
\begin{equation}
  \label{philb}
\|\phi_{\Ge}\|_{\mathfrak{H}_{p}} \ge c_{p}(x_{0})\Ge^{1-\frac{1}{p}}  
\end{equation}
for all sufficiently small $\Ge$.
\end{theorem}
The proof is in Appendix~\ref{app:phiHp}. It is based on the fact that the
solution $\phi_{\Ge}$ of the $\phi$-problem, given by (\ref{phipsi}) and
(\ref{phirep}), is expressed in terms of the explicitly known eigenfunctions
$u(x;\mu)$, given by (\ref{uofx}), of the integral operator $K$.

We can now complete the circle of inequalities and prove (\ref{goal}). 
According to Theorem~\ref{th:phiHp}, $(\phi_{\Ge}^{+},\phi_{\Ge}^{-})$
is an admissible pair for the $(f,g)$-problem, where $\Ge$ is
replaced with $(C_{p}/c_{p}(x_{0}))\Ge^{\frac{1}{p}}$, permitting us to conclude that
$\overline{\Gg}(x_{0})/p\le\Gg_{*}(x_{0})$,
where
\[
\underline{\Gg}(x_{0})=\limi_{\Ge\to 0}\frac{\ln\GD^{x_{0}}(\Ge)}{\ln\Ge},\quad
\overline{\Gg}(x_{0})=\lims_{\Ge\to 0}\frac{\ln\GD^{x_{0}}(\Ge)}{\ln\Ge}.
\]
Combining this inequality with inequalities (\ref{pstUB}) and (\ref{gammaLB}), we get
\begin{equation}
  \label{chain}
  \Gg_{*}(x_{0}^{p})\le\underline{\Gg_{p}}(x_{0})\le\overline{\Gg_{p}}(x_{0})\le
\overline{\Gg}(x_{0})\le p\Gg_{*}(x_{0}),
\end{equation}
where
\[
\underline{\Gg_{p}}(x_{0})=\limi_{\Ge\to 0}\frac{\ln\GD_{p}^{x_{0}}(\Ge)}{\ln\Ge},\quad
\overline{\Gg_{p}}(x_{0})=\lims_{\Ge\to 0}\frac{\ln\GD_{p}^{x_{0}}(\Ge)}{\ln\Ge}.
\]
The explicit form of $\Gg_{*}(x_{0})$, given by (\ref{gammastar}) implies that
$\Gg_{*}(x_{0})$ is a continuous function of $x_{0}$. Then, passing to the limit as
$p\to 1^{+}$ in (\ref{chain}), we obtain the existence of the limits
\[
\lim_{p\to1^{+}}\underline{\Gg_{p}}(x_{0})=\lim_{p\to1^{+}}\overline{\Gg_{p}}(x_{0})=\Gg_{*}(x_{0}).
\]
Inequality (\ref{gammaLB}) then implies that
\[
\underline{\Gg_{p}}(x_{0})\le\underline{\Gg}(x_{0})\le\overline{\Gg}(x_{0})\le p\Gg_{*}(x_{0}).
\]
Passing to the limit in this inequality as $p\to 1^{+}$ proves the existence of the limit
\[
\Gg(x_{0})=\lim_{\Ge\to 0}\frac{\ln\GD^{x_{0}}(\Ge)}{\ln\Ge},
\]
as well as the desired equality (\ref{goal}).


\section{The local problem}
\setcounter{equation}{0} 
\label{sec:local}
Suppose $f_{0}\in\mathfrak{C}_{2}$ is given, as well as $x_{0}\ge 1$. Let
\[
\CK_{\Ge}[f_{0}]=\{f\in\mathfrak{C}_{2}:\|f-f_{0}\|_{2}\le\Ge\}.
\]
We note that $\CK_{\Ge}[f_{0}]$ is a convex set.
The goal is to compute 
\begin{equation}
  \label{localf}
  M_{\Ge}(x_{0};f_{0})=\max_{f\in \CK_{\Ge}[f_{0}]}f(x_{0}),\qquad
m_{\Ge}(x_{0};f_{0})=\min_{f\in \CK_{\Ge}[f_{0}]}f(x_{0}).
\end{equation}
While the Kuhn-Tucker theorem is applicable to the
local problem (\ref{localf}) and leads to optimality conditions that are easy
to check numerically, they are not very useful as a guide for finding
the extremals in (\ref{localf}).

For this reason, we forgo the details of the Kuhn-Tucker-based analysis and opt
instead for the direct variational approach due to Caprini
\cite{capr74,capr80,capr81}, which is narrower in scope than
Kuhn-Tucker, but leads directly to a natural algorithm for computing the extremals in
(\ref{localf}) approximately. The method is applicable for the minimization of
general positive definite quadratic functionals, and necessitates the dual
reformulation of the variational problems (\ref{localf}). Given $f_{0}\in\mathfrak{C}_{2}$ and
$\Gd\in(-\Gd_{-},\Gd_{+})$, for some small $\Gd_{\pm}>0$, we seek to solve
\begin{equation}
  \label{Caprmin}
  \min_{\myatop{f\in\mathfrak{C}_{2}}{f(x_{0})-f_{0}(x_{0})=\Gd}}\|f-f_{0}\|_{2}^{2}.
\end{equation}
Suppose that
$f_{\Gs_{*}}$ satisfies the constraint $f_{\Gs_{*}}(x_{0})-f_{0}(x_{0})=\Gd$
and minimizes the functional $J[\Gs]=\|f_{\Gs}-f_{0}\|_{2}^{2}$.
The Caprini method is based on the following representation of the variation
$\GD J=J[\Gs]-J[\Gs_{*}]\ge 0$:
\begin{equation}
  \label{Caprep}
  \GD J=\|f_{\Gs}-f_{0}\|_{2}^{2}-\|f_{\Gs_{*}}-f_{0}\|_{2}^{2}=2\int_{0}^{\infty}C(t)d\GD\Gs(t)+\|f_{\GD\Gs}\|_{2}^{2}, 
\end{equation}
where $\GD\Gs=\Gs-\Gs_{*}$, and
\begin{equation}
  \label{Caprfn}
  C(t)=(\GL f_{\Gs_{*}})(t)-(\GL f_{0})(t)=\int_{0}^{1}e^{-xt}(f_{\Gs_{*}}(x)-f_{0}(x))dx
\end{equation}
is the Caprini function. 
\begin{theorem}
  The minimizer $\Gs_{*}$ in (\ref{Caprmin}) exists and is unique and has
  either a finite support or a countable support $\{t_{n}:n\ge 1\}$ with
\begin{equation}
  \label{MSz}
 \sum_{n=1}^{\infty}\nth{t_{n}}<\infty.
\end{equation}
In either case
\begin{equation}
  \label{Coptim}
C(t)\ge\frac{e^{-x_{0}t}}{f_{0}(x_{0})+\Gd}\int_{0}^{1}f_{\Gs_{*}}(x)(f_{\Gs_{*}}(x)-f_{0}(x))dx,
\quad t\ge 0,
\end{equation}
with equality at all $t=t_{n}$ in the support of $\Gs_{*}$. Conversely, if
$\Gs_{*}$ is a positive measure, whose support $\{t_{n}:n\ge 1\}$ satisfies
(\ref{MSz}), and is such that (\ref{Coptim}) holds, then it is a minimizer in
(\ref{Caprmin}), provided $\Gs_{*}\not=\Gs_{0}$, where $f_{\Gs_{0}}=f_{0}$.
\end{theorem}
\begin{proof}
To prove existence, we let $\Gs_{n}$ be a minimizing sequence. Then the
boundedness of $\|f_{\Gs_{n}}\|_{2}$ implies the boundedness of
$\|\Gs_{n}\|_{*}$, according to Lemma~\ref{lem:L2star}. Hence, we can extract
a subsequence, not relabeled, such that $f_{\Gs_{n}}\weak f_{*}$ in
$L^{2}(0,1)$ and $\Gs_{n}\wk{*}\Gs_{*}$ in $X^{*}$, where $X$ is given by
(\ref{X}). Then $f_{\Gs_{n}}(x)\to f_{\Gs_{*}}(x)$ for all $x>0$
since $e^{-xt}\in X$ for all $x>0$. We conclude that $f_{*}=f_{\Gs_{*}}$, and
that $f_{\Gs_{*}}(x_{0})-f_{0}(x_{0})=\Gd$. The weak lower semicontinuity of
the $L^{2}(0,1)$ norm implies that 
\[
\|f_{*}-f_{0}\|_{2}\le\limi_{n\to\infty}\|f_{\Gs_{n}}-f_{0}\|_{2}=
\min_{\myatop{f\in\mathfrak{C}_{2}}{f(x_{0})-f_{0}(x_{0})=\Gd}}\|f-f_{0}\|_{2}^{2}.
\]
The uniqueness of the minimizer follows from the convexity of the constraint
and the strict uniform convexity of the $L^{2}(0,1)$ norm.

Now, let $\Gs_{*}$ be the minimizer in (\ref{Caprmin}).
Assume first that $\Gs_{*}$ has a point mass
at $t_{*}$. Then, we remove $\Ge\Gd_{t_{*}}(t)$ from $\Gs_{*}$,
while placing the mass $\Ge e^{x_{0}(t_{0}-t_{*})}$ at $t_{0}$,
preserving the constraint. In that case
\[
\GD J=2\Ge (e^{x_{0}(t_{0}-t_{*})}C(t_{0})-C(t_{*}))+O(\Ge^{2})\ge 0,
\]
and therefore, $e^{x_{0}t_{0}}C(t_{0})\ge e^{x_{0}t_{*}}C(t_{*})$. Hence, any
point mass $t_{*}$ in the support of $\Gs_{*}$ must be a point of global
minimum of $e^{x_{0}t}C(t)$ on $[0,+\infty)$. If $t_{*}$ is in the support of $\Gs_{*}$, but is
not a point mass, then $m(\Ge)=\Gs_{*}((t_{*}-\Ge)^{+},t_{*}+\Ge))\to 0$, as
$\Ge\to 0$, while $m(\Ge)>0$ for any $\Ge>0$. In that case, we remove
$\Gs_{*}|((t_{*}-\Ge)^{+},t_{*}+\Ge)$ from $\Gs_{*}$ and place the appropriate
mass $m(\Ge)e^{x_{0}(t_{0}-t_{*})}$ at $t_{0}$, so as to maintain the
constraint. This time, we obtain
\[
\GD J=m(\Ge)(e^{x_{0}(t_{0}-t_{*})}C(t_{0})-C(t_{*}))+o(m(\Ge)).
\]
Once again, we conclude that $t_{*}$
must be a point of global minimum of $e^{x_{0}t}C(t)$. Since
$e^{x_{0}t}C(t)$ is an entire function of $t$, as is evident from
(\ref{Caprfn}), the support of
$\Gs_{*}$ must be discrete. If the support of $\Gs_{*}$ is infinite
\begin{equation}
  \label{sigmastar}
  \Gs_{*}=\sum_{n=1}^{\infty}a_{n}\Gd_{t_{n}}(t),\quad a_{n}>0,
\end{equation}
and does not satisfy (\ref{MSz}), then, by the M\"untz-Szasz
theorem \cite{fell68}, the set of functions
$u^{t_{n}}$ are dense in $C_{0}([0,1])$. But then the
functions $e^{-xt_{n}}$ are dense in $C_{0}(0,\infty)$. In that case
the equation
\begin{equation}
  \label{minC}
  e^{x_{0}t_{n}}C(t_{n})=m\defeq\min_{t\ge 0}e^{x_{0}t}C(t)
\end{equation}
would imply that
\[
\int_{0}^{1}g(x)(f_{*}(x)-f_{0}(x))dx=g(x_{0})m,\quad\forall
g\in C_{0}(0,\infty),
\]
where $f_{*}$ is a shorthand for $f_{\Gs_{*}}$.
This easily leads to a contradiction if, for example, we take a
delta-like sequence $g_{n}(x)$ converging to $\Gd_{a}(x)$ for an arbitrary
$a\in(0,1)$.

Now, equation (\ref{minC}) written as $C(t_{n})=e^{-x_{0}t_{n}}m$ implies
\[
\int_{0}^{1}f_{*}(x)(f_{*}(x)-f_{0}(x))dx=f_{*}(x_{0})m,
\]
giving a formula for $m$,
\begin{equation}
  \label{minCform}
  m=\nth{f_{*}(x_{0})}\int_{0}^{1}f_{*}(x)(f_{*}(x)-f_{0}(x))dx.
\end{equation}
The constraint $f_{*}(x_{0})-f_{0}(x_{0})=\Gd$, can then be incorporated into the
optimality conditions by replacing $f_{*}(x_{0})$ by
$f_{0}(x_{0})+\Gd$ in (\ref{minCform}), obtaining (\ref{Coptim}).

To see that (\ref{Coptim}) with equality provision
is sufficient for optimality, we integrate (\ref{Coptim}) with respect to
$\Gs_{*}$, and obtain $f_{0}(x_{0})+\Gd=f_{*}(x_{0})$, taking (\ref{Caprfn})
into account, unless
\begin{equation}
  \label{badcase}
  \int_{0}^{1}f_{*}(x)(f_{*}(x)-f_{0}(x))dx=0.
\end{equation}
However, if (\ref{badcase}) holds, then (\ref{Coptim}) reads $C(t)\ge
0$. Integrating this inequality with respect to $\Gs_{0}$, such that
$f_{0}=f_{\Gs_{0}}$, we obtain
\begin{equation}
  \label{f0ineq}
  \int_{0}^{1}f_{0}(x)(f_{*}(x)-f_{0}(x))dx\ge 0.
\end{equation}
Subtracting (\ref{f0ineq}) from (\ref{badcase}), we obtain $\|f_{0}-f_{*}\|\le
0$, which implies that $f_{*}=f_{0}$ and hence, $\Gs_{*}=\Gs_{0}$. This shows
that (\ref{Coptim}) implies $f_{*}(x_{0})-f_{0}(x_{0})=\Gd$, provided $\Gs_{*}\not=\Gs_{0}$.

Now, if $\Gs$ is any competitor
measure, satisfying the constraint, then equation (\ref{Caprep})
becomes
\[
\GD J=2\left(\int_{0}^{\infty}C(t)d\Gs(t)-f_{*}(x_{0})m\right)+\|f_{\GD\Gs}\|_{2}^{2},
\]
where
\[
m=\frac{1}{f_{0}(x_{0})+\Gd}\int_{0}^{1}f_{\Gs_{*}}(x)(f_{\Gs_{*}}(x)-f_{0}(x))dx.
\]
Discarding $\|f_{\GD\Gs}\|_{2}^{2}$ and using inequality
(\ref{Coptim}), we obtain
\[
\GD J\ge 2\left(m \int_{0}^{\infty}e^{-x_{0}t}d\Gs(t) -f_{*}(x_{0})m\right)=2m(f_{\Gs}(x_{0})-f_{*}(x_{0}))=0
\]
since, due to the constraint, we must have 
$f_{\Gs}(x_{0})=f_{*}(x_{0})$ for any competitor measure.

\end{proof}

To illustrate the optimality conditions, let us consider an example with
$f_{0}(x)=e^{-x}$. In this case, the solutions of (\ref{Caprmin}) can be
computed explicitly. The forms of these solutions were, in fact, suggested by
first solving these problems numerically with an algorithm based on formula
(\ref{Caprep}). If $\Gd>0$, then $f_{*}(x)=f^{+}_{*}(x)=a+be^{-x\tau}$ for
appropriately chosen $a>0$, $b>0$ and $\tau>1$. If $\Gd\in(-e^{-x_{0}},0)$,
then $f_{*}(x)=f^{-}_{*}(x)=ae^{-x\tau}$ for appropriately chosen $a>0$ and
$\tau>1$. If $\Gd>0$, the optimality condition (\ref{Coptim}) gives equations
$\Hat{C}(0)=\hat{C}(\tau)=0$, and $\Hat{C}'(\tau)=0$, where
$\Hat{C}(t)=C(t)-me^{-x_{0}t}$. Together with the constraint,
$f_{*}(x_{0})=f_{0}(x_{0})+\Gd$, this results in 4 equations for the 4
unknowns $a,b,\tau$ and $m$. Similarly, if $\Gd<0$, the optimality condition
(\ref{Coptim}) gives equations $\hat{C}(\tau)=0$, and $\Hat{C}'(\tau)=0$,
which together with the constraint, results in 3 equations for the 3 unknowns
$a,\tau$ and $m$. The resulting system of equations is linear in $(a,b,m)$ for
$\Gd>0$ and in $(a,m)$ for $\Gd<0$, so that these parameters can be easily
eliminated, leading to a single algebraic equation for $\tau$. This equation
is very complicated to be displayed here, but it can be easily investigated
either numerically or by means of a computer algebra system and shown to have
a unique solution $\tau(x_{0},\Gd)$, for all $x_{0}\ge 1$ and
$\Gd\in(-e^{-x_{0}},0)$, if $\Gd<0$, and $\Gd\in(0,+\infty)$, if $\Gd>0$.
When $\Ge$ is small, we find
\[
M_{\Ge}(x_{0};e^{-x})=E_{+}(x_{0})\Ge+O(\Ge^{2}),\quad
m_{\Ge}(x_{0};e^{-x})=E_{-}(x_{0})\Ge+O(\Ge^{2}),
\]
where $E_{+}(x_{0})$ is an increasing function of $x_{0}$ from
\[
E_{+}(1)=\sqrt{-\frac{e^{4}-8e^{3}+14e^{2}+8e-19}{(e^{2}-2e-1)(3e^{2}-10e+5)}}\approx 2.67788263
\]
to
\[
E_{+}(\infty)=\sqrt{-\frac{e^{2}+2e-1}{3e^{2}-10e+5}}\approx 27.488747597.
\]
The function $E_{-}(x_{0})$ behaves in a more complicated manner. It increases
from 
\[
E_{-}(1)=2\sqrt{\frac{e^{2}-1}{e^{4}-6e^{2}+1}}\approx 1.5
\]
to its maximal value $E_{-}((e+2)/(e+1))\approx 1.566$ and then decreases to 0
as $x_{0}$ increases from $(e+2)/(e+1)$ to $+\infty$, In fact,
$\hat{E}_{-}(x_{0})=x_{0}^{-1}e^{x_{0}}E_{-}(x_{0})$ is a monotone increasing
function from $eE_{-}(1)\approx 4.1$ to 
\[
\hat{E}_{-}(\infty)=2e\sqrt{\frac{2(e^{2}-1)}{(e^{2}-1)^{2}-4e^{2}}}\approx 5.8.
\]
The plots of $f^{\pm}_{*}(x)$ and $f_{0}(x)$ together with their respective certificates of
optimality $\Hat{C}(t)=C(t)-me^{-x_{0}t}$ are shown in Fig.~\ref{fig:true}.
\begin{figure}[t]
  \centering
  \includegraphics[scale=0.32]{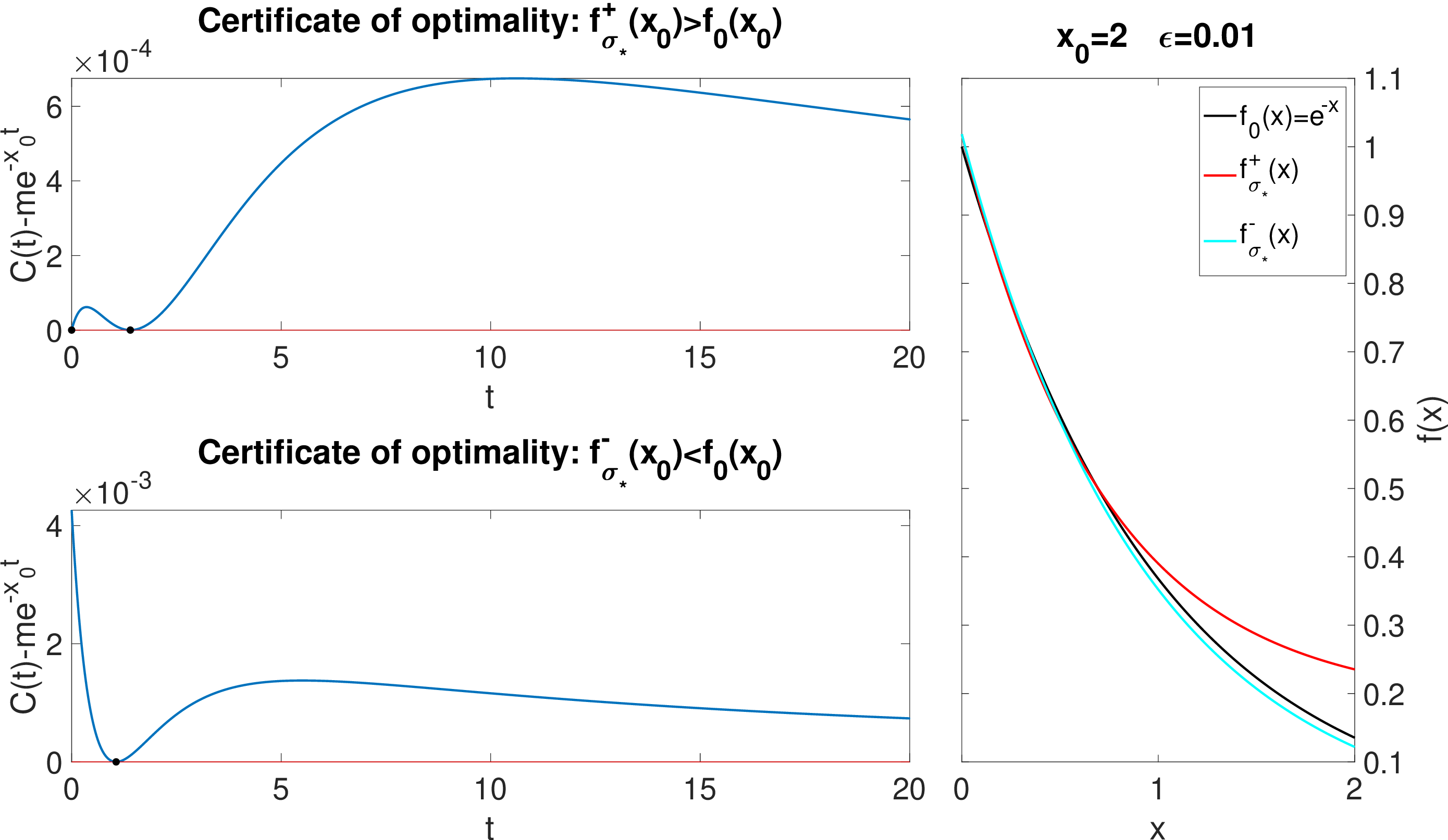}
  \caption{Solutions of the local worst case extrapolation problems (\ref{Caprmin})
    with $f_{0}(x)=e^{-x}$, $x_{0}=2$, $\Ge=0.01$, and their respective
  certificates of optimality.}
  \label{fig:true}
\end{figure}
\medskip

\textbf{Acknowledgments.} This material is based upon work supported
by the National Science Foundation under Grant No. DMS-2305832. The
authors are grateful to David Grabovsky, for referring them to papers
on the asymptotics of the Gauss hypergeometric function.

\def\cprime{$'$} \ifx \cedla \undefined \let \cedla = \c \fi\ifx \cyr
  \undefined \let \cyr = \relax \fi\ifx \cprime \undefined \def \cprime
  {$\mathsurround=0pt '$}\fi\ifx \prime \undefined \def \prime {'}
  \fi\def\Ya{Ya}

\appendix

\section{Kuhn-Tucker in topological vector spaces}
\setcounter{equation}{0}
\label{app:KT}
Let $X$ be a  locally convex topological vector space. Let $\CF\subset
X^{*}\oplus\bb{R}$ be any subset. Define
\begin{equation}
  \label{Kdef}
  K=\{x\in X: f(x)\le\Ga\ \forall(f,\Ga)\in\CF\}.
\end{equation}
Then $K\subset X$ is both closed and convex. Let $h\in X^{*}$ be a given functional.
The maximization problem
\begin{equation}
  \label{LP}
  m=\sup_{x\in K}h(x)
\end{equation}
is called the linear programming problem. If the set $K$ is empty the value of
$m$ is set to $-\infty$ by convention. 

Let $\Hat{\CF}$ denote the smallest closed (in weak-* topology of $X^{*}\oplus\bb{R}$) convex cone
containing $\CF$. We remark that
\[
K=\{x\in X: f(x)\le\Ga\ \forall(f,\Ga)\in\Hat{\CF}\}.
\]
We also define 
\[
K^{*}=\{(f,\Ga)\in X^{*}\oplus\bb{R}:f(x)\le\Ga\ \forall x\in K\}.
\]
Obviously, $\Hat{\CF}\subset K^{*}$. It is easy to give an example
where $K^{*}\not=\Hat{\CF}$. Let $X=\bb{R}$ and $\CF=\{(1,0)\}$, so
that $K=\{x\in\bb{R}: x\le 0\}$ and
$\Hat{\CF}=\{(f,0)\in\bb{R}^{2}:f\ge 0\}$. But
\[
K^{*}=\{(f,\Ga)\in\bb{R}^{2}:fx\le\Ga\ \forall x\le
0\}=\{(f,\Ga)\in\bb{R}^{2}:f\ge 0,\ \Ga\ge 0\}.
\]

Our goal is to obtain a dual formulation of (\ref{LP}). We observe that if
$m<+\infty$, then $(h,m)\in K^{*}$, while $(h,m-\Ge)\not\in K^{*}$ for
any $\Ge>0$. Thus, $m$ is the smallest of the numbers $\Ga$, such that
$(h,\Ga)\in K^{*}$. For this reason, we introduce the following notation. For
any subset $S\subset X^{*}\times\bb{R}$ and any $f\in X^{*}$, we define
\[
S_{f}=\{\Ga\in\bb{R}:(f,\Ga)\in S\}.
\]
Our remark can then be stated that $m<+\infty$ \IFF $K^{*}_{h}\not=\emptyset$,
in which case $m=\min K^{*}_{h}$.  The dual set $K^{*}$ is a maximal set of
inequalities defining $K$, while the set $\Hat{\CF}\subset K^{*}$ describes
the weak-* closure of the set of inequalities obtained by positive linear
combinations of finite subsets of inequalities in (\ref{Kdef}). The remarkable fact of the
Kuhn-Tucker theorem is that even though $\Hat{\CF}$ can be a lot smaller than
$K^{*}$, as our example showed, it still contains all the bottom extremal
points of $K^{*}$.
\begin{theorem}
  \label{th:KT}
Suppose that the set $K$, given by (\ref{Kdef}), is non-empty. Let $\Hat{\CF}$
be the smallest weak-* closed convex cone
containing $\CF$. Let $m$ be given by (\ref{LP}). Then
\begin{equation}
  \label{dualLP}
m=\min\Hat{\CF}_{h},
\end{equation}
where we have indicated that the minimum is achieved, if $\Hat{\CF}_{h}\not=\emptyset$.
\end{theorem}
We remark that requiring $K\not=\emptyset$ is essential.
For example, we can take $X=\bb{R}^{2}$ and
$\CF=\{(\Be_{1},0),(-\Be_{1},-1)\}$, corresponding to constraints $x_{1}\le 0$
and $-x_{1}\le -1$, which are inconsistent, so that $K=\emptyset$. We compute
\[
\Hat{\CF}=\{((\Gl_{1}-\Gl_{2})\Be_{1},-\Gl_{2}):\Gl_{1}\ge 0,\ \Gl_{2}\ge 0\}.
\]
For $h=\Be_{2}$ the set of pairs $(\Be_{2},\Ga)\in\Hat{\CF}$ is empty
resulting in the minimum in (\ref{dualLP}) to be $+\infty$, while the supremum
over the empty set is $-\infty$.
\begin{proof}
We have already observed that
\[
\sup_{x\in K}h(x)<+\infty\quad\eqv\quad K^{*}_{h}\not=\emptyset.
\]
Therefore, if $K^{*}_{h}=\emptyset$ then $\Hat{\CF}_{h}=\emptyset$
since $\Hat{\CF}\subset K^{*}$. Thus, if $m=+\infty$, then formula
(\ref{dualLP}) is valid. It only remains to consider the case
$m<+\infty$, whereby $(h,m)\in K^{*}$. The theorem below asserts that
$(h,m)\in\Hat{\CF}$, and therefore, that $m$ has to be equal to the \rhs\ of
(\ref{dualLP}) since $(h,m-\Ge)\not\in K^{*}$ for every $\Ge>0$.
\end{proof}

\begin{theorem}
  \label{preKT}
  Under assumptions of Theorem~\ref{th:KT} assume additionally that
  $m<+\infty$. Then $(h,m)\in\Hat{\CF}$.
\end{theorem}
\begin{proof}
  If $(h,m)\not\in\Hat{\CF}$ then, by the Hahn-Banach convex separation
  theorem there exists $\xi_{0}\in X$, $\mu_{0}\in\bb{R}$, $\Gg\in\bb{R}$,
  such that
  \begin{equation}
    \label{HB}
    h(\xi_{0})+\mu_{0}m<\Gg\le f(\xi_{0})+\mu_{0}\Ga,\quad\forall(f,\Ga)\in\Hat{\CF}.
  \end{equation}
Here, we used the fact that the set of all linear continuous functionals on
$X^{*}$, equipped with its weak-* topology is parametrized by $X$, i.e. for
any $F\in(X^{*},\text{weak-*})^{*}$ there exists a unique $x\in X$, such that
$F(f)=f(x)$ for all $f\in X^{*}$.

We first observe that if there exists $(f_{0},\Ga_{0})\in\Hat{\CF}$, such that
$f_{0}(\xi_{0})+\mu_{0}\Ga_{0}<0$ then the second inequality in (\ref{HB})
cannot hold since $(\Gl f_{0},\Gl\Ga_{0})\in\Hat{\CF}$ for any
$\Gl>0$. However, if $f_{0}(\xi_{0})+\mu_{0}\Ga_{0}\ge 0$ then $\Gl
f_{0}(\xi_{0})+\mu_{0}\Gl\Ga_{0}$ can be made as close to 0 as one wishes. It
follows that $\Gg=0$. We thus restate (\ref{HB}) in a more convenient form:
\begin{equation}
  \label{HBcone}
  h(\xi_{0})+\mu_{0}m<0,\qquad f(\xi_{0})+\mu_{0}\Ga\ge 0,\quad\forall(f,\Ga)\in\Hat{\CF}.
\end{equation}
We need to consider 3 possibilities for $\mu_{0}$.
\begin{enumerate}
\item $\mu_{0}>0$. In this case
\[
f\left(-\frac{\xi_{0}}{\mu_{0}}\right)\le\Ga,\quad\forall(f,\Ga)\in\Hat{\CF}.
\]
which implies that $-\xi_{0}/\mu_{0}\in K$. But then, according to the first
inequality in (\ref{HBcone}),
\[
h\left(-\frac{\xi_{0}}{\mu_{0}}\right)>m,
\]
which contradicts the definition (\ref{LP}) of $m$.
\item $\mu_{0}=0$. Since $K\not=\emptyset$ there exists $u\in K$. But then for any
  $\Gl\ge 0$, we have
\[
f(u-\Gl\xi_{0})\le\Ga,\quad\forall(f,\Ga)\in\Hat{\CF}.
\]
This implies that $u-\Gl\xi_{0}\in K$. But $h(u-\Gl\xi_{0})=h(u)-\Gl
h(\xi_{0})$, which can be made arbitrarily large and positive by a choice of
$\Gl>0$ since $h(\xi_{0})<0$. This contradicts the assumption that $m<+\infty$.
\item $\mu_{0}<0$. For convenience of working with positive numbers, we set
  $\mu_{0}=-\nu_{0}$, and $\nu_{0}>0$. In that case, we have
  $f(\xi_{0})\ge\nu_{0}\Ga$ for every $(f,\Ga)\in\Hat{\CF}$. Then for every
  $x\in K$, we have for any $t>0$
\[
f(x-t\xi_{0})\le(1-t\nu_{0})\Ga.
\]
Thus, for all $x\in K$ and $t\in(0,1/\nu_{0})$, we conclude that
\[
y(x,t)=\frac{x-t\xi_{0}}{1-t\nu_{0}}\in K.
\]
We will get a contradiction by showing that
\[
\sup_{\myatop{x\in K}{0<t<\nu_{0}^{-1}}}h(y(x,t))>m.
\]
We compute
\[
h(y(x,t))=m+\frac{h(x)-m-t(h(\xi_{0})-\nu_{0}m)}{1-t\nu_{0}}.
\]
By definition of the supremum there exist $x_{0}\in K$, such that
\[
h(x_{0})>m+\frac{h(\xi_{0})-\nu_{0}m}{2\nu_{0}}
\]
since $h(\xi_{0})-\nu_{0}m<0$. But then $h(y(x_{0},(2\nu_{0})^{-1}))>m$.
\end{enumerate}
The obtained contradictions imply that $(h,m)\in\Hat{\CF}$, establishing (\ref{dualLP}).
\end{proof}


\section{Asymptotics of $u(z;\mu)$ for large $\mu$}
\setcounter{equation}{0}
\label{app:HG}
To compute the asymptotics of $u(z;\mu)$, as $\mu\to+\infty$ for $z\in\GO=\{z\in\bb{C}:\Re
z>0,\ z\not\in[0,1]\}$, we first apply the Pfapf transformation
\cite[formula (1.8)]{khda14}, and obtain
\[
  u(z;\mu)=\nth{z}F\left(\left[\frac{1}{4}+\frac{i\mu}{2}, \frac{1}{4}-\frac{i\mu}{2}\right],[1];1-\nth{z^{2}}\right).
\]
We note that The map $g(z)=1-z^{-2}$ maps $\GO$ into
$\Hat{\GO}=\bb{C}\setminus\{w\in\bb{R}:w(w-1)\ge 0\}$, to which the
asymptotic expansion from \cite[Theorem 3.2]{khda14}
applies. Substituting our parameters into the expansion \cite[(3.8)--(3.11)]{khda14}
and retaining only the leading term ($n=1$ in the expansion), we obtain
\begin{multline*}
\pi iF\left(\left[\frac{1}{4}+\frac{i\mu}{2},\frac{1}{4}-\frac{i\mu}{2}\right],
[1];1-\nth{z^{2}}\right)\sim\\
\left(\frac{\xi}{2}\right)^{\frac{1}{2}}
\left(e^{\left(\frac{i\mu}{2}-\frac{1}{4}\right)\pi i}K_{-\frac{1}{2}}\left(-\frac{i\xi\mu}{2}\right)
-e^{\left(\frac{3}{4}-\frac{i\mu}{2}\right)\pi i}K_{-\frac{1}{2}}\left(\frac{i\xi\mu}{2}\right)\right)c_{0}+O(\Phi_{1}(\mu,\xi)),
\end{multline*}
where
\[
\xi=\ln\left(1-\frac{2}{z^{2}}-2i\sqrt{\left(1-\nth{z^{2}}\right)\nth{z^{2}}}\right),\qquad
c_{0}=-\frac{\sqrt{z}}{(z^{2}-1)^{1/4}}.
\]
\begin{multline*}
  \Phi_{1}(\mu,\xi)=
e^{-\frac{\pi\mu}{2}}\frac{\sqrt{|\xi|}}{\mu}\left|K_{-\frac{1}{2}}\left(-\frac{i\mu\xi}{2}\right)\right|
+e^{\frac{\pi\mu}{2}}\frac{\sqrt{|\xi|}}{\mu}\left|K_{-\frac{1}{2}}\left(\frac{i\mu\xi}{2}\right)\right|\\
+\frac{e^{-\frac{\pi\mu}{2}}}{\sqrt{|\xi|}\mu}\left|K_{\frac{1}{2}}\left(-\frac{i\mu\xi}{2}\right)\right|
+\frac{e^{\frac{\pi\mu}{2}}}{\sqrt{|\xi|}\mu}\left|K_{\frac{1}{2}}\left(\frac{i\mu\xi}{2}\right)\right|.
\end{multline*}
Here the transformation $\Gz=1-2z^{-2}$ maps 
\[
\GO=\{z\in\bb{C}:\Re z>0,\ z\not\in[0,1]\}
\]
onto
\[
G=\bb{C}\setminus\{\Gz\in\bb{R}:|\Gz|\ge 1\}.
\]
Then
\[
\xi=\ln(\Gz-i\sqrt{1-\Gz^{2}}).
\]
We observe that $\Gz=\cosh\xi$, and therefore, $\xi(\Gz)$ is injective
on $G$. Thus, $\Md_{\infty}\xi(G)=\xi(\Md_{\infty}G)$. Computing the
images of $(-\infty,-1]\pm 0i$ and $[1,+\infty)\pm 0i$ and noting that
$\xi(\infty)=\infty$, we conclude that $\xi(\Gz)$
maps $G$ onto the strip $-\pi<\Im\xi<0$ bijectively. We also note that
\[
\cosh\xi=\Gz=1-\nth{z^{2}},
\]
which implies that
\[
\nth{z^{2}}=-\sinh^{2}\left(\frac{\xi}{2}\right).
\]
Since $z\in\GO$ lies in the right half-plane, while $\Im\xi\in(-\pi,0)$, we conclude
that
\[
\nth{z}=i\sinh\left(\frac{\xi}{2}\right)=\sin\left(\frac{i\xi}{2}\right).
\]
We can write this as
\[
\cos\left(\frac{\pi}{2}-\frac{i\xi}{2}\right)=\nth{z},\quad-\pi<\Im\xi<0.
\]
Since the map $\eta=\pi/2-i\xi/2$ maps the strip $-\pi<\Im\xi<0$ onto
the strip $\Re\eta\in(0,\pi/2)$, we conclude that $\eta=\Ga(z)$,
where $\Ga(z)$ was defined in (\ref{Ralpha}) and, thus,
\begin{equation}
  \label{zxi}
  \xi=i(2\Ga(z)-\pi). 
\end{equation}
Using (\ref{zxi}) and the formulas
\[
K_{\frac{1}{2}}(z)=K_{-\frac{1}{2}}(z)=\sqrt{\frac{\pi}{2}}\frac{e^{-z}}{\sqrt{z}},
\]
we obtain the error estimate
\[
  O(\Phi_{1}(\mu,\xi))=O\left(\frac{e^{\frac{\pi\mu}{2}(1+\Im\xi/\pi)}}{\mu\sqrt{\mu}}\right)
  =O\left(\frac{|e^{\pi\mu\Ga(z)}|}{\mu\sqrt{\mu}}\right).
\]
Since $|\Im\xi|<\pi$, we conclude that the term
$e^{\left(\frac{i\mu}{2}-\frac{1}{4}\right)\pi i}K_{-\frac{1}{2}}\left(-\frac{i\xi\mu}{2}\right)$ 
is negligible, compared to
$e^{\left(\frac{3}{4}-\frac{i\mu}{2}\right)\pi i}K_{-\frac{1}{2}}\left(\frac{i\xi\mu}{2}\right)$. 
Therefore, we obtain the asymptotics 
\[
u(z;\mu)\sim\frac{e^{i\pi/4}}{\sqrt{2\pi\mu}}
\frac{e^{\frac{\pi\mu}{2}\left(1-\frac{i\xi}{\pi}\right)}}{(z^{2}-1)^{1/4}\sqrt{z}}
\frac{\sqrt{\xi}}{\sqrt{i\xi}}
+O\left(\frac{|e^{\pi\mu\Ga(z)}|}{\mu\sqrt{\mu}}\right).
\]
Since $-\pi<\Im\xi<0$, we conclude that
\[
\frac{\sqrt{\xi}}{\sqrt{i\xi}}=e^{-\frac{i\pi}{4}}.
\]
Thus, for all $z\in\GO$
\begin{equation}
  \label{uzmuasymp}
  u(z;\mu)=\frac{1}{\sqrt{2\pi\mu}}\frac{e^{\pi\mu\Ga(z)}}{(z^{2}-1)^{1/4}\sqrt{z}}
  +O\left(\frac{|e^{\pi\mu\Ga(z)}|}{\mu\sqrt{\mu}}\right).
\end{equation}


\section{Estimate of $\|\phi_{\Ge}\|_{\mathfrak{H}_{p}}$}
\setcounter{equation}{0}
\label{app:phiHp}
The goal of this section is to prove the lower bound (\ref{philb}) on
$\|\phi_{\Ge}\|_{\mathfrak{H}_{p}}$. When $x_{0}>1$,
Part (i) of Theorem~\ref{th:psiyp} can be used to estimate
$\|\psi_{\Gve}\|_{\mathfrak{H}_{p}}$ from below. If $x_{0}=1$
\begin{equation}
  \label{psieps1}
\psi_{\Gve}(z)=\int_{0}^{\infty}\frac{u(z;\mu)\mu\tanh(\pi\mu)}{2\hat{\Gve}^{2}\cosh(\pi\mu)+1}d\mu.  
\end{equation}
Its asymptotics as $\Gve\to 0^{+}$ is given by the following theorem.
\begin{theorem}
\label{th:psiz1}
Let $z \in \GO=\{z\in\bb{C}:\Re z>0,\ z\not\in[0,1]\}$, and
$\psi_{\Gve}$ be the solution of the integral equation (\ref{inteq}) with $x_{0}=1$. Then
\begin{equation}
  \label{psiz1}
\psi_{\Gve}(z)\sim\frac{R(z)\sqrt{|\ln\hat{\eps}|}}{\pi\sin(\Ga(z))}\hat{\eps}^{\frac{-2\Ga(z)}{\pi}},\quad
\hat{\eps}=\frac{\Gve}{\sqrt{2\pi}}.
\end{equation}
where $R(z)$ and $\Ga(z)$ are defined in (\ref{Ralpha}).
\end{theorem}
\begin{proof}
As we have argued before, the asymptotics of $\psi_{\Gve}(z)$ is determined by the asymptotics
of the integrand in (\ref{psieps1}), as $\mu\to\infty$. Thus, we would want to replace
$u(z;\mu)$ by its asymptotics (\ref{uofx0}), $\tanh(\pi\mu)$ by 1, and
$2\cosh(\pi\mu)$ by $e^{\pi\mu}$. We therefore, rewrite (\ref{psieps1}) as
\begin{equation}
\label{psi1z}
\psi_{\Gve}(z)=\frac{R(z)}{\sqrt{2\pi}}\int_{0}^{\infty}\frac{\sqrt{\mu} e^{\Ga(z)\mu}v(z;\mu)}{2\hat{\Gve}^2\cosh(\pi \mu) + 1} d\mu,
\end{equation}
where
\[
v(z;\mu) = \frac{u(z;\mu)}{u_0(z;\mu)}\tanh(\pi \mu),
\]
and where $u_{0}(z;\mu)$ is given by (\ref{u0zmu}). Then, $v(z;\cdot) \in
C([0,\infty))$, due to the representation (\ref{Euler}), as argued in the
proof of Theorem~\ref{th:psiyp}, and $v(z;\mu)\to 1$, as $\mu\to\infty$, by
Lemma~\ref{lem:ux0asym}. Thus, there exists $M(z)>0$, such that $|v(z;\mu)|\le
M(z)$, for any $z\in\GO$.

Let $I(\hat{\eps})$ denote the integral in (\ref{psi1z}). Changing variables
by $\mu'=\pi \mu+2\ln(\hat{\eps})$, we obtain 
\[
I(\hat{\eps})=\nth{\pi}\sqrt{\frac{-2 \ln \hat{\eps}}{\pi}} \hat{\eps}^{\frac{-2 \Ga(z)}{\pi}}
\int_{2 \ln \hat{\eps}}^{\infty} \sqrt{\frac{\mu'-2 \ln \hat{\eps}}{ -2 \ln \hat{\eps}}}
\left( \frac{ e^{\frac{\Ga(z) \mu'}{\pi}} v(z;\frac{\mu'}{\pi} - \frac{2 \ln \hat{\eps}}{\pi})}{e^{\mu'} + e^{-\mu' + 4 \ln \hat{\eps}} + 1} \right) d\mu'
\]
The estimate
\[
\left| \sqrt{\frac{\mu'-2 \ln \hat{\eps}}{ -2 \ln \hat{\eps}}}
\left( \frac{ e^{\frac{\Ga(z) \mu'}{\pi}} v(z;\frac{\mu'}{\pi} - \frac{2 \ln \hat{\eps}}{\pi})}{e^{\mu'} + e^{-\mu' + 4 \ln \hat{\eps}} + 1} \right) \right| \chi_{(2\ln \hat{\eps}, \infty)}(\mu')
\le \Phi(\mu')
\]
where $\Phi(\mu')$ is given by
\[
\Phi(\mu')  = 
\begin{cases} 
   M(z)e^{\left(\frac{\Re\Ga(z)}{\pi} - 1\right)\mu'} & \mu' < 0, \\
   M(z)\sqrt{\mu' + 1} e^{\left(\frac{\Re\Ga(z)}{\pi} - 1\right)\mu'} &\mu' > 0,
   \end{cases}
\]
shows that the Lebesgue dominated convergence theorem is applicable since
$\Re\Ga(z)\in(0,\pi/2)$, by Lemma~\ref{lem:expbnd}. Therefore,

\[
\psi_{\Gve}(z)\sim\frac{R(z)}{\pi^{2}}\sqrt{|\ln\hat{\Gve}|}\hat{\eps}^{\frac{-2 \Ga(z)}{\pi}}
\int_{\mathbb{R}}\frac{e^{\frac{\Ga(z) \mu'}{\pi}}}{e^{\mu'} + 1}d\mu'
= \frac{R(z)\sqrt{|\ln\hat{\Gve}|}}{\pi\sin(\Ga(z))}\hat{\eps}^{\frac{-2 \Ga(z)}{\pi}}.
\]
The theorem is proved.
\end{proof}

In order to estimate $\|\psi_{\Ge}\|_{\mathfrak{H}_{p}}$ (for any $x_{0}\ge
1$), we need a tighter bound on
$\Re\Ga((iy)^{1/p})$, when $y>0$ and $p>1$, which becomes optimal as
$p\to 1^{+}$. Formula (\ref{Ralpha}) show that $\Re\Ga(iy+0)=\pi/2$, for any $y>0$. In fact,
we have the following estimate.
\begin{lemma}
\label{lem:iyp}
Let $y>0$ and $p>1$. Then
$\Re \alpha((iy)^{1/p}) \in \left(\frac{\pi}{2p}, \frac{\pi}{2}\right)$.
\end{lemma}
\begin{proof}
We first observe that for any $z\in\GO$
\[
\Ga(z)=-i\ln z+i\ln(1-i\sqrt{z^{2}-1}).
\]
Indeed, it is easy to see that the \rhs\ of the above formula is analytic in
$\GO$ and agrees with $\arccos(1/z)$ for $z>1$. The same is true for the
\lhs. Therefore, they must agree everywhere in $\GO$.
If $z=(iy)^{1/p}$, then $z^{2}=re^{i\Gth_{p}}$, where $\Gth_{p}=\pi/p\in(0,\pi)$, and $r>0$. It
is now easy to see that $\arg(z^{2}-1)$, as a function of $r$, decreases from $\pi$ at $r=0$ to
$\Gth_{p}$ at $r=+\infty$. Hence, $\arg(-i\sqrt{z^{2}-1})$ decreases from 0 at $r=0$ to
$\Gth_{p}/2-\pi/2$ at $r=+\infty$. Therefore, $\arg(1-i\sqrt{z^{2}-1})$
will also be between 0 and $\Gth_{p}/2-\pi/2$. Thus, 
\[
\Re\Ga(z)=\frac{\Gth_{p}}{2}-\arg(1-i\sqrt{z^{2}-1})\in\left(\frac{\Gth_{p}}{2},\frac{\pi}{2}\right).
\]
\end{proof}

\begin{theorem}
  \label{th:psiub}
For $x_0\ge 1$ and $p>1$, there is a constant $s_p(x_{0})>0$ such that 
\[
\| \psi_{\Gve}\|_{\mathfrak{H}_p} \geq s_p(x_{0})
\begin{cases}
  \Gve^{-\frac{2\Ga(x_{0})}{\pi}-\frac{1}{p}},&x_{0}>1,\\
  \Gve^{-\frac{1}{p}}\sqrt{|\ln\eps|},&x_{0}=1.
\end{cases}
\]
for all sufficiently small $\Gve>0$.
\end{theorem}
\begin{proof}
Let
\[
\psi^{x_{0}}_{\Gve}(z)=
\begin{cases}
\frac{R(x_0)R(z)}{2\pi\sin(\pi\Gb(z))}\hat{\Gve}^{-2\Gb(z)},&x_{0}>1,\\
  \frac{R(z)\sqrt{|\ln\hat{\eps}|}}{\pi\sin(\Ga(z))}\hat{\eps}^{\frac{-2\Ga(z)}{\pi}},&x_{0}=1,
\end{cases}\qquad\Gb(z)=\frac{\Ga(x_{0})+\Ga(z)}{\pi}.
\]
Then, Theorems~\ref{th:psiyp}(i) and \ref{th:psiz1} say that
$\psi_{\Gve}(z)\sim\psi^{x_{0}}_{\Gve}(z)$ for any $z\in\GO$ and any $x_{0}\ge
1$. We then write
\[
  \|\psi_{\Gve}\|_{\mathfrak{H}_{p}}^{2}=
  \nth{\pi}\int_{0}^{\infty}\frac{|\psi_{\Gve}((iy)^{1/p})|^{2}}{N_{p}(y)|\psi^{x_{0}}_{\Gve}((iy)^{1/p})|^{2}}
  |\psi^{x_{0}}_{\Gve}((iy)^{1/p})|^{2}dy,
\]
where $N_{p}(y)=y^{\frac{p-1}{p}}|1+(iy)^{1/p}|^{2}$.
By Lemma~\ref{lem:iyp}, we estimate
\begin{equation}
  \label{psi0lb}
    |\psi^{x_{0}}_{\Gve}((iy)^{1/p})|\ge A_{p}(x_{0},y)K_{p}^{x_{0}}(\Gve),
  \end{equation}
where 
\[
A_{p}(x_{0},y)=
\begin{cases}
  \frac{R(x_0)|R((iy)^{1/p})|(2\pi)^{\frac{\Ga(x_{0})}{\pi}+\frac{1}{2p}}}{2\pi|\sin(\pi\Gb((iy)^{1/p}))|},&x_{0}>1,\\
\frac{|R((iy)^{1/p})|(2\pi)^{\frac{1}{2p}}}{\pi|\sin(\pi\Ga((iy)^{1/p}))|},&x_{0}=1,
\end{cases}\quad
K_{p}^{x_{0}}(\Gve)=
\begin{cases}
    \Gve^{-\frac{2\Ga(x_{0})}{\pi}-\frac{1}{p}},&x_{0}>1,\\
    \Gve^{-\frac{1}{p}}\sqrt{|\ln\Gve|},&x_{0}=1.
  \end{cases}
\]
Thus, we obtain the lower bound
\[
\|\psi_{\Gve}\|_{\mathfrak{H}_{p}}^{2}\ge\frac{K_{p}^{x_{0}}(\Gve)^{2}}{\pi}
\int_{0}^{\infty}\frac{|\psi_{\Gve}((iy)^{1/p})|^{2}}{N_{p}(y)|\psi^{x_{0}}_{\Gve}((iy)^{1/p})|^{2}}
  A_{p}(x_{0},y)^{2}dy.
\]
Now, by Fatou's lemma, we have, taking into account
$\psi_{\Gve}(z)\sim\psi^{x_{0}}_{\Gve}(z)$, as $\Gve\to 0^{+}$,
\[
\limi_{\Gve\to 0}\frac{\|\psi_{\Gve}\|_{\mathfrak{H}_{p}}^{2}}{K_{p}^{x_{0}}(\Gve)^{2}}\ge 
\frac{1}{\pi}\int_{0}^{\infty}\frac{A_{p}(x_{0},y)^{2}}{N_{p}(y)}dy=:2s_{p}(x_{0})^{2}>0.
\]
It follows that for all sufficiently small $\Gve>0$, we have
$\|\psi_{\Gve}\|_{\mathfrak{H}_{p}}\ge s_{p}(x_{0})K_{p}^{x_{0}}(\Gve)$.
\end{proof}
We now have everything we need to prove Theorem~\ref{th:phiHp}.
\begin{proof}[Proof of Theorem~\ref{th:phiHp}]
We have, using (\ref{GeGve})
\begin{equation}
  \label{phiepsHp}
  \|\phi_{\Ge}\|_{\mathfrak{H}_{p}}=\frac{\Ge\|\psi_{\Gve}\|_{\mathfrak{H}_{p}}}{\|\psi_{\Gve}\|_{2}}
=\frac{\|\psi_{\Gve}\|_{\mathfrak{H}_{p}}}{\|\psi_{\Gve}\|}.
\end{equation}
It only remained to observe that Theorems~\ref{th:psiyp}(iii),
\ref{th:psi1}(ii) can be written as
\[
\|\psi_{\Gve}\|\sim C_{0}(x_{0})K_{p}^{x_{0}}(\Gve)\Gve^{\frac{1}{p}-1},
\]
where
\[
C_{0}(x_{0})=
\begin{cases}
  \frac{(2\pi)^{\Gb(x_{0})/2}}{\pi}\sqrt{\frac{x_0\arccos(1/x_{0})}{2(x_0^2-1)}},&x_{0}>1,\\
\frac{\sqrt{2}}{\pi},&x_{0}=1.
\end{cases}
\]
Combining this with Theorem~\ref{th:psiub} and applying to (\ref{phiepsHp}),
we obtain that
\[
\|\phi_{\Ge}\|_{\mathfrak{H}_{p}}\ge
\frac{s_{p}(x_{0})}{2C(x_{0})}\Gve^{1-\frac{1}{p}}
\]
for all sufficiently small $\Gve>0$. The theorem is now proved.
\end{proof}


\begin{thebibliography}{10}

\bibitem{bazh15}
Emilia Bazhlekova.
\newblock Completely monotone functions and some classes of fractional
  evolution equations.
\newblock {\em Integral Transforms Spec. Funct.}, 26(9):737--752, 2015.

\bibitem{bern29}
Serge Bernstein.
\newblock Sur les fonctions absolument monotones.
\newblock {\em Acta Mathematica}, 52(1):1--66, 1929.

\bibitem{capr74}
I.~Caprini.
\newblock On the best representation of scattering data by analytic functions
  in ${L}\sb{2}$-norm with positivity constraints.
\newblock {\em Nuovo Cimento A (11)}, 21:236--248, 1974.

\bibitem{capr80}
I.~Caprini.
\newblock General method of using positivity in analytic continuations.
\newblock {\em Rev. Roumaine Phys.}, 25(7):731--740, 1980.

\bibitem{capr81}
I.~Caprini.
\newblock Constraints on physical amplitudes derived from a modified analytic
  interpolation problem.
\newblock {\em J. Phys. A}, 14(6):1271--1279, 1981.

\bibitem{ciulli69}
S.~Ciulli.
\newblock A stable and convergent extrapolation procedure for the scattering
  amplitude.---{I}.
\newblock {\em Il Nuovo Cimento A (1965-1970)}, 61(4):787--816, Jun 1969.

\bibitem{davis52}
Philip Davis.
\newblock An application of doubly orthogonal functions to a problem of
  approximation in two regions.
\newblock {\em Transactions of the American Mathematical Society},
  72(1):104--137, 1952.

\bibitem{kron26}
R.~de~L.~Kronig.
\newblock On the theory of dispersion of {X}-rays.
\newblock {\em Josa}, 12(6):547--557, 1926.

\bibitem{deto18}
Laurent Demanet and Alex Townsend.
\newblock Stable extrapolation of analytic functions.
\newblock {\em Foundations of Computational Mathematics}, 19(2):297--331, 2018.

\bibitem{dhjg13}
Richard~D Dortch, Kevin~D Harkins, Meher~R Juttukonda, John~C Gore, and Mark~D
  Does.
\newblock Characterizing inter-compartmental water exchange in myelinated
  tissue using relaxation exchange spectroscopy.
\newblock {\em Magnetic resonance in medicine}, 70(5):1450--1459, 2013.

\bibitem{dgvd10}
Adrienne~N Dula, Daniel~F Gochberg, Holly~L Valentine, William~M Valentine, and
  Mark~D Does.
\newblock Multiexponential {T2}, magnetization transfer, and quantitative
  histology in white matter tracts of rat spinal cord.
\newblock {\em Magnetic Resonance in Medicine: An Official Journal of the
  International Society for Magnetic Resonance in Medicine}, 63(4):902--909,
  2010.

\bibitem{ener97}
J{\"o}rg Enderlein and Rainer Erdmann.
\newblock Fast fitting of multi-exponential decay curves.
\newblock {\em Optics Communications}, 134(1):371--378, 1997.

\bibitem{khda14}
S~Farid~Khwaja and AB~Olde~Daalhuis.
\newblock Uniform asymptotic expansions for hypergeometric functions with large
  parameters iv.
\newblock {\em Analysis and Applications}, 12(06):667--710, 2014.

\bibitem{fell39}
W~Feller.
\newblock Completely monotone functions and sequences.
\newblock {\em Duke Math. J}, 5(3):661--674, 1939.

\bibitem{fell68}
William Feller.
\newblock On m{\"u}ntz'theorem and completely monotone functions.
\newblock {\em The American Mathematical Monthly}, 75(4):342--350, 1968.

\bibitem{fran90}
Joel Franklin.
\newblock Analytic continuation by the fast fourier transform.
\newblock {\em SIAM journal on scientific and statistical computing},
  11(1):112--122, 1990.

\bibitem{gmm16}
Roberto Garrappa, Francesco Mainardi, and Guido Maione.
\newblock Models of dielectric relaxation based on completely monotone
  functions.
\newblock {\em Fractional Calculus and Applied Analysis}, 19(5):1105--1160,
  2016.

\bibitem{grab_Stielt}
Yury Grabovsky.
\newblock Reconstructing {S}tieltjes functions from their approximate values: a
  search for a needle in a haystack.
\newblock {\em SIAM J. Appl. Math.}, 82(4):1135--1166, 2022.

\bibitem{grho-annulus}
Yury Grabovsky and Narek Hovsepyan.
\newblock Explicit power laws in analytic continuation problems via reproducing
  kernel {H}ilbert spaces.
\newblock {\em Inverse Problems}, 36(3):035001, 2020.

\bibitem{grho-CEMP}
Yury Grabovsky and Narek Hovsepyan.
\newblock On feasibility of extrapolation of the complex electromagnetic
  permittivity function using {K}ramers-{K}ronig relations.
\newblock {\em SIAM J. Math Anal.}, 53(6):6993--7023, 2021.

\bibitem{grho-gen}
Yury Grabovsky and Narek Hovsepyan.
\newblock Optimal error estimates for analytic continuation in the upper
  half-plane.
\newblock {\em Comm Pure Appl Math}, 71:140--170, 2021.
\newblock to appear.

\bibitem{gpch00}
Jessie Greener, Hartwig Peemoeller, Changho Choi, Rick Holly, Eric~J. Reardon,
  Carolyn~M. Hansson, and Mik~M. Pintar.
\newblock Monitoring of hydration of white cement paste with proton nmr
  spin--spin relaxation.
\newblock {\em Journal of the American Ceramic Society}, 83(3):623--627, 2000.

\bibitem{gps03}
Bj{\"o}rn Gustafsson, Mihai Putinar, and Harold~S Shapiro.
\newblock Restriction operators, balayage and doubly orthogonal systems of
  analytic functions.
\newblock {\em Journal of Functional Analysis}, 199(2):332--378, 2003.

\bibitem{hausdorff21}
Felix Hausdorff.
\newblock Summationsmethoden und momentfolgen. i, ii.
\newblock {\em Mathematische Zeitschrift}, 9:74--109, 280--299, 1921.
\newblock 10.1007/BF01378337.

\bibitem{isvy99}
Andrei~A Istratov and Oleg~F Vyvenko.
\newblock Exponential analysis in physical phenomena.
\newblock {\em Review of Scientific Instruments}, 70(2):1233--1257, 1999.

\bibitem{kato16}
Alexander Katsevich and Alexander Tovbis.
\newblock Diagonalization of the finite hilbert transform on two adjacent
  intervals.
\newblock {\em Journal of Fourier Analysis and Applications}, 22(6):1356--1380,
  2016.

\bibitem{kimb74}
Clark~H Kimberling.
\newblock A probabilistic interpretation of complete monotonicity.
\newblock {\em Aequationes mathematicae}, 10:152--164, 1974.

\bibitem{kramers27}
H.~A. Kramers.
\newblock La diffusion de la lumiere par les atomes.
\newblock {\em Atti. del Congresso Internazionale dei Fisici}, 2:545--557,
  1927.

\bibitem{Krein:1977:MMP}
M.~G. Krein and A.~A. Nudelman.
\newblock {\em The Markov Moment Problem and Extremal Problems}.
\newblock American Mathematical Society, Providence, RI, 1977.

\bibitem{krnu98}
MG~Krein and AA~Nudelman.
\newblock An interpolation problem in the class of {S}tieltjes functions and
  its connection with other problems.
\newblock {\em Integral Equations and Operator Theory}, 30(3):251--278, 1998.

\bibitem{loan19}
R.~J. Loy and R.~S. Anderssen.
\newblock Approximation of and by completely monotone functions.
\newblock {\em ANZIAM J.}, 61(4):416--430, 2019.

\bibitem{merk14}
Milan Merkle.
\newblock Completely monotone functions: A digest.
\newblock In Gradimir~V. Milovanovi{\'c} and Michael~Th. Rassias, editors, {\em
  Analytic Number Theory, Approximation Theory, and Special Functions}, pages
  347--364. Springer, 2014.

\bibitem{mill70}
Keith Miller.
\newblock Least squares methods for ill-posed problems with a prescribed bound.
\newblock {\em SIAM Journal on Mathematical Analysis}, 1(1):52--74, 1970.

\bibitem{misa01}
Kenneth~S Miller and Stefan~G Samko.
\newblock Completely monotonic functions.
\newblock {\em Integral Transforms and Special Functions}, 12(4):389--402,
  2001.

\bibitem{nies77}
H.~Niessner.
\newblock Multiexponential fitting methods.
\newblock In Jean Descloux and J{\"u}rg Marti, editors, {\em Numerical
  Analysis: Proceedings of the Colloquium on Numerical Analysis Lausanne,
  October 11--13, 1976}, pages 63--76. Birkh{\"a}user Basel, Basel, 1977.

\bibitem{nfjk19}
Dmitry~S Novikov, Els Fieremans, Sune~N Jespersen, and Valerij~G Kiselev.
\newblock Quantifying brain microstructure with diffusion mri: Theory and
  parameter estimation.
\newblock {\em NMR in Biomedicine}, 32(4):e3998, 2019.

\bibitem{parfenov}
O~G Parfenov.
\newblock Asymptotics of singular numbers of imbedding operators for certain
  classes of analytic functions.
\newblock {\em Mathematics of the {USSR}-Sbornik}, 43(4):563--571, apr 1982.

\bibitem{pesc10}
Victor Pereyra and Godela Scherer.
\newblock {\em Exponential data fitting and its applications}.
\newblock Bentham Science Publishers, 2010.

\bibitem{pute17}
Mihai Putinar and James~E Tener.
\newblock Singular values of weighted composition operators and second
  quantization.
\newblock {\em International Mathematics Research Notices},
  2018(20):6426--6441, 2017.

\bibitem{rate13}
Adithya Rajan and Cihan Tepedelenlio\u{g}lu.
\newblock A representation for the symbol error rate using completely monotone
  functions.
\newblock {\em IEEE Trans. Inform. Theory}, 59(6):3922--3931, 2013.

\bibitem{rlfs09}
David~A. Reiter, Ping-Chang Lin, Kenneth~W. Fishbein, and Richard~G. Spencer.
\newblock Multicomponent {T2} relaxation analysis in cartilage.
\newblock {\em Magnetic Resonance in Medicine}, 61(4):803--809, 2009.

\bibitem{svhf02}
Y.-Q. Song, L.~Venkataramanan, M.D. H{\"u}rlimann, M.~Flaum, P.~Frulla, and
  C.~Straley.
\newblock T1--t2 correlation spectra obtained using a fast two-dimensional
  laplace inversion.
\newblock {\em Journal of Magnetic Resonance}, 154(2):261--268, 2002.

\bibitem{stilt1894}
T.-J. Stieltjes.
\newblock Recherches sur les fractions continues.
\newblock {\em Annales de la Facult{\'e} des sciences de Toulouse: 1er
  s{\'e}rie Math{\'e}matiques}, 8(4):J1--J122, 1894.

\bibitem{trefe19}
Lloyd~N Trefethen.
\newblock Quantifying the ill-conditioning of analytic continuation.
\newblock {\em BIT Numerical Mathematics}, 60(4):901--915, 2020.

\bibitem{vese99}
Sergio Vessella.
\newblock A continuous dependence result in the analytic continuation problem.
\newblock {\em Forum Mathematicum}, 11(6):695--703, 1999.

\bibitem{widd34}
David~V Widder.
\newblock The inversion of the laplace integral and the related moment problem.
\newblock {\em Transactions of the American Mathematical Society},
  36(1):107--200, 1934.

\bibitem{widd38}
David~V Widder.
\newblock The stieltjes transform.
\newblock {\em Transactions of the American Mathematical Society}, 43(1):7--60,
  1938.

\bibitem{widd31}
DV~Widder.
\newblock Necessary and sufficient conditions for the representation of a
  function as a laplace integral.
\newblock {\em Transactions of the American Mathematical Society},
  33(4):851--892, 1931.

\bibitem{zast19}
V.~P. Zastavny\u{\i}.
\newblock Some problems related to completely monotone and positive definite
  functions.
\newblock {\em Mat. Zametki}, 106(2):222--240, 2019.

\end{thebibliography}
\end{document}